\documentclass[10pt,twoside,english,reqno,a4paper]{amsart}
\usepackage{listings,graphicx,amsmath,varioref,amscd,amssymb,color,bm,stmaryrd}
\usepackage[colorlinks=false]{hyperref}
\usepackage{amsthm}
\usepackage{esint}
\usepackage[margin=2in]{geometry}

\usepackage{fouriernc}

\newlength{\defbaselineskip}
\newcommand{\R}{\mathbb{R}}

\newcommand{\dist}{\mbox{dist}}

\newcommand{\car}{{\raise4pt\hbox{$\chi$}}}

\newcommand{\sop}{{\rm supp\,}}

\newcommand{\Div}{\hbox{\rm div\,}}

\def\rn{\mathbb{R}^{N}}

\def\into{\int_{\Omega}}

\newcommand{\DM }{\mathcal{DM}^\infty }
\newcommand{\res}{\!\!\mathop{\hbox{
			\vrule height 7pt width .5pt depth 0pt
			\vrule height .5pt width 6pt depth 0pt}}
	\nolimits}

	\long\def\salta#1{\relax}

\usepackage{amssymb}
\usepackage{amsmath}
\usepackage{amsthm}
\usepackage{amscd}
\usepackage{amssymb}
\usepackage{amsfonts}
\usepackage{graphics}

\usepackage{hyperref}
\usepackage{graphics}

\usepackage{color}

\theoremstyle{plain}
\newtheorem{theorem}{Theorem}[section]

\newtheorem{Proposition}[theorem]{Proposition}
\newtheorem{lemma}[theorem]{Lemma}

\theoremstyle{definition}
\newtheorem{defin}[theorem]{Definition}

\newtheorem{remark}[theorem]{Remark}
\newtheorem{example}{Example}

\theoremstyle{remark}

\def\bk{\color{black}}
\def\dys{\displaystyle}

\begin{document}
\title{The Dirichlet problem for singular elliptic equations with general nonlinearities}

\author[V. De Cicco]{Virginia De Cicco}
\author[D. Giachetti]{Daniela Giachetti}
\author[F. Oliva]{Francescantonio Oliva}
\author[F. Petitta]{Francesco Petitta}

\address{Dipartimento di Scienze di Base e Applicate per l' Ingegneria, ``Sapienza" Universit\`a di Roma, Via Scarpa 16, 00161 Roma, Italia}
\address[Virginia De Cicco]{virginia.decicco@sbai.uniroma1.it}
\address[Daniela Giachetti]{daniela.giachetti@sbai.uniroma1.it}
\address[Francescantonio Oliva]{francesco.oliva@sbai.uniroma1.it}
\address[Francesco Petitta]{francesco.petitta@sbai.uniroma1.it}


\maketitle

\begin{abstract}
In this paper, under very general assumptions, we prove existence and  regularity of distributional solutions to homogeneous Dirichlet problems of the form 
$$\begin{cases}
					 \displaystyle - \Delta_{1} u = h(u)f &  \text{in}\, \Omega, \\
					 u\geq 0& \text{in}\ \Omega,\\
					 u=0 & \text{on}\ \partial \Omega,			
				 \end{cases}
$$
where, $\Delta_{1} $ is the $1$-laplace operator, $\Omega$ is a bounded open subset of $\R^N$ with Lipschitz boundary,   $h(s)$ is a continuous function  which may become singular at $s=0^{+}$, and $f$ is a nonnegative datum in $L^{N,\infty}(\Omega)$ with suitable small norm.  Uniqueness of solutions is also shown  provided  $h$ is decreasing and $f>0$. As a by-product of our method a general  theory for the same problem involving the $p$-laplacian as principal part, which is missed in the literature, is established. The main assumptions we use are also further discussed in order to show their optimality. 
\end{abstract}
\tableofcontents
\section{Introduction}
The main goal of this paper is to deal with existence, regularity (and uniqueness, if attainable) of distributional solutions to problems involving the $1$-laplacian as a principal part and a  lower order terms which can become singular on the set where the solution $u$ vanishes. Formally the problem looks like
 \begin{equation} \label{pb1lapintro}
				 \begin{cases}
					 \displaystyle -\Delta_1 u := - \operatorname{div}\left(\frac{D u}{|D u|}\right)= h(u)f &  \text{in}\, \Omega, \\
					 u\geq 0& \text{in}\ \Omega,\\
					 u=0 & \text{on}\ \partial \Omega,			
				 \end{cases}
			 \end{equation}
where $\Omega$ is a bounded open subset of $\R^N$ with Lipschitz boundary,  $0 \leq f\in L^{N,\infty}(\Omega)$, $h(s)$ is a nonnegative continuous function defined in $[0,\infty)$, bounded at infinity, possibly singular at $s=0$ (i.e. $h(0)=\infty$) without any monotonicity property.

The natural space for this kind of problems is  $BV$ (or its local version $BV_{\rm loc}$), the space of functions of bounded variation, i.e. the space of $L^1$ functions whose gradient is a Radon measure with finite (or locally finite) total variation. In \eqref{pb1lapintro} $\frac{D u}{|D u|}$ is the Radon-Nikodym derivative of the measure $Du$ with respect to its total variation $|Du|$.

Problems involving the $1$-laplace operator, which is known to be closely related to the mean curvature operator (\cite{OsSe}),  enter in  a variety of both practical and theoretical  issues as for instance  in image restoration and in  torsion problems (\cite{K, ka, Sapiro, M, BCRS}). We refer the interested reader to the monograph \cite{ACM} for a more complete review on applications.

\medskip 

From the mathematical point of view, in order to give sense to the left hand side term of the equation in \eqref{pb1lapintro}, Andreu, Ballester, Caselles and Maz\'on  (\cite{ABCM})  proposed to use  the Anzellotti theory of pairings $(z, Du)$ of $L^\infty$--divergence--measure vector fields $z$ and the gradient of a $BV$ function $u$. In their definition the  vector field $z$ belongs to $\mathcal{D}\mathcal{M}^\infty(\Omega)$ (the space of $L^\infty$ vector fields whose distributional divergence is a Radon measure with bounded total variation), and is  such that both $\|z\|_\infty\le 1$ and $(z, Du)=|Du|$; in this way  the vector $z$ is intended to play the role of the ratio $\frac{D u}{|D u|}$.

\medskip 

Problem \eqref{pb1lapintro} has been studied by several authors when $h\equiv 1$ (see \cite{CT,KS, MST1} and references therein) under suitable smallness assumption on the datum $f$. The non-autonomous case has also been treated; the eigenvalue problem is discussed in \cite{KS}, the absorption case is handled  in \cite{MP}, while  the (sub-)critical exponent problem has been also addressed (see for instance \cite{D, MS}). Finally, some early results can be found in \cite{DGS} concerning  the mild singular model case  $h(s)={s^{-\gamma}}$, $0<\gamma\leq 1$.  As one of our main point also concerns the singular case of a nonlinearity  $h(s)$ which is  unbounded near the origin $s=0$, it is expected  that,  if $f$ vanishes in a portion of $\Omega$ of positive Lebesgue measure, then also the region where $u$ degenerates at zero plays a non-trivial role; in this case, in fact,   the characteristic function $\chi_{\{u>0\}}$ may appear in the definition of solution of problem \eqref{pb1lapintro} (see Definition \ref{weakdefnonnegative} in  Section \ref{sec:fnonnegativa} below). 

\medskip

To be more transparent and to fix the ideas, we  assume, at first,  that  $f$ is a strictly positive function in $L^{N}(\Omega)$;  most of the  main issues  are already present in this less general case in which, although, some  further  properties for solutions to \eqref{pb1lapintro} can be proven  easing the overall  the presentation. In this case,  for instance, $h(u)f\in L^1(\Omega)$ and uniqueness holds, in a suitable class of functions, if  $h$ is decreasing.

Let us  explain the meaning of solution in this particular non-degenerate case. If $0<f\in L^{N}(\Omega)$,  a function $u\in BV_{\rm{loc}}(\Omega)\cap L^\infty(\Omega)$ will be a distributional   solution to problem \eqref{pb1lapintro} if there exists $z\in \mathcal{D}\mathcal{M}^\infty(\Omega)$ with $||z||_{\infty}\le 1$ such that
					$-\operatorname{div}z =h(u)f \ \ \ \text{in  } \mathcal{D'}(\Omega)$, $(z,Du)=|Du|$ as measures in  $\Omega$, and one of the following conditions holds: \begin{equation}					\lim_{\epsilon\to 0}  \fint_{\Omega\cap B(x,\epsilon)} u (y) dy = 0 \ \ \ \text{or} \ \ \ [z,\nu] (x)= -1 \label{def_bordointro}\ \ \ \text{for  $\mathcal{H}^{N-1}$-a.e. } x \in \partial\Omega.				\end{equation}	

					Condition \eqref{def_bordointro}
		is a (very weak) way to give sense to the boundary condition $u=0$ at $\partial\Omega$ in this case where $u$ is only in $BV_{\rm{loc}}(\Omega)\cap L^{\infty}(\Omega)$, and  where $[z,\nu]$ is the weak normal  trace of $z$ defined in \cite{An} (see Section \ref{sec2}  below).

In order to describe the core of our results, we will also  assume that  $h:[0,\infty)\to [0,\infty]$  is a continuous function, finite outside the origin, such that $h(0)\not =0$, 
	 	 \begin{equation*}\label{h1intro}
	 	 \displaystyle \exists\;{c_1},\gamma,k_0>0\;\ \text{such that}\;\  h(s)\le \frac{c_1}{s^\gamma} \ \ \text{if} \ \ s<k_0,
	 	 \end{equation*}
		 and
			 \begin{equation*}\label{h2intro}
				 \lim_{s\to \infty} h(s):=h(\infty)<\infty\,.
			 \end{equation*}
More general right hand sides  (for \eqref{pb1lapintro}) $F(x, u)$ are considered in Section \ref{more} (see Theorem \ref{existencenonnegativefs}) in order to include more general growth as, for instance,  
$$
\dys h(s)\sim {\rm exp}\left({\rm exp} \left(\frac{1}{s}\right)\right), 
$$ 	
as well as the general non-autonomous case. 

\medskip 

As far as the right hand side is then concerned in \eqref{pb1lapintro}, in order to get existence, we require a smallness condition on the datum that $0 < f\in L^{N}(\Omega)$ depending on the behavior of the function $h$ at infinity, i.e. \begin{equation}\label{smallintro}\displaystyle ||f||_{L^N(\Omega)}< \frac{1}{\mathcal{S}_1 h(\infty)}\,,\end{equation} where $\mathcal{S}_1$ is the best constant in the Sobolev inequality for functions in $W^{1,1}_{0}(\Omega)$. In particular, \eqref{smallintro} encodes the fact  that no smallness condition is assumed if $h(\infty)=0$. This shows that the behavior of $h$ at infinity can induce a  first regularizing effect on the problem if compared  to the non-singular case (e.g. $h\equiv 1$).  In fact, if $h(\infty)$ is a suitable  positive constant then  the smallness assumption \eqref{smallintro}  (or its general version for data in $L^{N,\infty}(\Omega)$, see Section \ref{lore}) is necessary and sharp as it can be deduced by comparison with the results in    \cite{CT} and \cite{MST1}. A second regularizing effect appears if $h(0)=\infty$;   the presence of a singular term  will imply that $u>0$ in $\Omega$ that is again in contrast with the non-singular case; we shall come back on this fact in a while.

We also remark that the equality  $(z,Du)=|Du|$ in the definition of solutions of \eqref{pb1lapintro} does not depend on $h$, which is in some sense natural if we think to the homogeneity of the principal part. However, this requires a formal proof (in particular, to handle the jump part of the pairing) which is based on a result in \cite{CDC}. Similarly, we are able to prove the weak boundary condition \eqref{def_bordointro} independently of the function $h$.

As we already mentioned,  this case of a  strictly positive datum $f$ enjoys particular features most of them summarized in  Theorem \ref{regularity} below. In particular one shall  see that 
$h(u)f \in L^1(\Omega)$ that,  together with the fact that $u^{\sigma} \in BV(\Omega)$ (with $\sigma= \max(1,\gamma)$), will imply uniqueness of  solutions, if we assume that $h$ is decreasing (see Theorem \ref{uniqueness}). Observe that uniqueness does not hold, in general, for merely nonnegative data (see \cite{DGS}).

\medskip 				

Let us briefly describe the technique we exploit in order to get existence. As noticed in \cite{ka}, one can  naturally handle with \eqref{pb1lapintro} via  approximation with  problems having $p$-laplacians  principal part  (with $p>1$) which in our case can look like
\begin{equation} \label{pbplapintro}
				 \begin{cases}
				 \displaystyle - \Delta_p u_p= h(u_p)f &  \text{in}\, \Omega, \\
								 u_p\geq 0& \text{in}\ \Omega,\\
					 u_p=0 & \text{on}\ \partial \Omega\,.			
				 \end{cases}
			 \end{equation}
 Note that, at the best of our knowledge,  even the existence of solutions $u_p$ for 
\eqref{pbplapintro} is still missed in the literature under this generality and this will require a preliminary study that we shall present in Section \ref{carrie} in the even more general case of a nonnegative $f\in L^{(p^{\ast})'}(\Omega)$.  In fact, existence of solutions $u_p$ of 
\eqref{pbplapintro} has been proven in the model case, i.e. $h(s)={s^{-\gamma}}$, $\gamma>0$,  (\cite{CST,DCA, DGS}), where, if $\gamma\le1$, $u_p$ is shown to have global finite energy in the sense that $u_p\in W^{1,p}_0(\Omega)$, as well as in the case $p=2$ (see \cite{BO, GMM2, OP}). On the other hand, if $\gamma> 1$, solutions $u_{p}$ have in general only locally finite energy, that is  $u_p\in W^{1,p}_{\rm{loc}}(\Omega)$, so that  the boundary datum needs to be assigned through a suitable weaker condition than the usual trace sense  (see \cite{BO, DCA, DCO, GMM, OP1, OP} for further remarks on this fact) eventually producing \eqref{def_bordointro} in the limit as $p\to 1^{+}$. Recently, existence of solutions to \eqref{pbplapintro}  has been shown in case of  a general function $h$ and a measure datum $f$ (\cite{DCDO}).

Note that this approximation approach with respect to  the parameter $p$ requires a priori estimates on the solutions $u_p$ of \eqref{pbplapintro} that are  independent of $p$, and that will be proven (see Section \ref{4.4}) provided \eqref{smallintro} holds. This will allow us to pass to the limit obtaining a solution to \eqref{pb1lapintro}. As we already mentioned,  if $h(0)=\infty$ then $u_{p}\longrightarrow u$ a.e. on $\Omega$ as $p\to 1^{+}$, and $u>0$; this  should be compared with the non-singular case in which, under assumption \eqref{smallintro}, $u_{p}\longrightarrow 0$ instead (see for instance \cite{CT}).

\medskip

In Section \ref{sec:fnonnegativa} we will show how our results extend  to the case of a general nonnegative  datum $f$.
Indeed, if $f$ can vanish on a subset of $\Omega$ of positive measure, the situation becomes much more delicate and, as we said,  it will  involve the region ${\{u>0\}}$ in an essential way.  
Among other technical points that will be discussed later, as in  \cite{DGS}, we will be forced to ask, in the definition of solutions to  \eqref{pbplapintro},  that  $\chi_{\{u>0\}} \in BV_{\rm{loc}}(\Omega)$ and that the equation 
\begin{equation}\label{anzintro}
-(\operatorname{div}z)\chi^*_{\{u>0\}} = h(u)f
\end{equation}
is satisfied in the sense of distributions, where $\chi_{\{u>0\}}$ is the characteristic function of the region ${\{u>0\}}$ and $\chi^*_{\{u>0\}}$ is its precise representative
(in the sense of the $BV$--function). This fact will lead to some technical complications in the proof of existence of a solution. Observe that the notion of solution  we use here will essentially coincide with the previous if $f>0$ (see Remark \ref{rema} below). 

\medskip 
We specify some further peculiarities of this case. First  notice that requiring    $\chi_{\{u>0\}}$ to be  a locally $BV$--function is equivalent to the fact that the region ${\{u>0\}}$ is a set of locally finite perimeter.

 Also, one may observe  that equation \eqref{anzintro} can also  be written as follows (see Remark \ref{remarkequazione} below)

\begin{equation*}
-\Div\big(z\chi_{\{u>0\}}\big)+|D \chi_{\{u>0\}}|=h(u)f,
\end{equation*}\bk
where the left hand side is a sum of an operator in divergence form and an additional term $|D \chi_{\{u>0\}}|$, which is a measure concentrated on the reduced boundary $\partial^*\{u>0\}$. Moreover, as $f$ is assumed to be merely  nonnegative, we are only able to prove that $h(u)f \in L^1_{\rm{loc}}(\Omega)$  (instead of $h(u)f \in L^1(\Omega)$ which holds true in the case of positive $f$)
and, as we mentioned, no uniqueness of solutions holds (see \cite{DGS}).

\medskip 
 
 In Section \ref{lore} we handle with $L^{N,\infty}$-data and here one needs to use the machinery of Lorentz spaces that will be briefly summarized for the convenience of the reader. The extension given in  this section is not only technical as it can be shown to be optimal (see Remark \ref{optimo} below).

\medskip 

Finally, in Section \ref{energiafinita}, we will also discuss the possibility to have finite energy solutions, i.e. $u\in BV(\Omega)$. As we said, global energy estimates  were only known in the mild model case (i.e. $\gamma\leq 1$). Although we get rid of this fact by working with suitable  compositions of the solutions, in the general framework of strong singularities the {\it global} $BV$ regularity of the solutions $u$  is still an open question;  Section \ref{energiafinita} will be devoted to give some partial answers and insights on this issue. We conclude by investigating the case where a more general right-hand side of the form $F(x,u)$ is considered (Section \ref{more}).

\section{Notations and Preliminaries}\label{sec2}
\setcounter{equation}{0}

In the entire paper $\mathcal H^{N-1}(E)$  denotes the $(N - 1)$-dimensional Hausdorff measure of a set $E$ while, for simplicity,  $|E|$ will stand for the classical  $N$-dimensional  Lebesgue measure. 
Here $\Omega$ is an open bounded subset of $\R^N$ ($N\ge 1$) with Lipschitz boundary.  We will denote by $\mathcal{M}(\Omega)$ the space of Radon measures with finite total variation over $\Omega$.

\medskip 
We denote by 
$$\DM(\Omega):=\{ z\in L^\infty(\Omega;\R^N) : \operatorname{div}z \in \mathcal{M}(\Omega) \},$$ 

and by $\DM_{\rm{loc}}(\Omega)$ the vector fields $z\in L^\infty(\Omega;\R^N)$  such that  $\operatorname{div}z\in \mathcal{M}(\omega)$,  $\forall \omega\subset\!\subset\Omega$\bk.
\\ As usual
$$BV(\Omega):=\{ u\in L^1(\Omega) : Du \in \mathcal{M}(\Omega, \R^N) \},$$ 
and its local counterpart, which is the space of functions $u\in BV(\omega)$ for all $\omega\subset \subset \Omega$, is denoted by $BV_{\rm{loc}}(\Omega)$.
We recall that for $BV(\Omega)$ a norm is given by 
$$ \|u\|_{BV(\Omega)}=\int_\Omega |u|\, + \int_\Omega|Du|\,,$$
or by 
$$\displaystyle \|u\|_{BV(\Omega)}=\int_{\partial\Omega}
|u|\, d\mathcal H^{N-1}+ \int_\Omega|Du|\,.$$
By $S_u$ we mean the set of $x\in\Omega$ such that $x$ is
not a Lebesgue point of $u$, by $J_u$ the jump set and by $u^* $ the precise representative of $u$.
For more properties regarding $BV$ spaces we refer to \cite{AFP}, from which we mainly derive our notations. We also refer to \cite{EG, Zi}.

The theory of $L^\infty$-divergence-measure vector fields is due
to Anzellotti \cite{An} and to Chen and
Frid \cite{CF}.   First of all it can be shown that if $z\in \DM(\Omega)$ then $\operatorname{div}z $ is absolutely continuous with respect to $\mathcal H^{N-1}$.

We define the following distribution $(z,Dv): C^1_c(\Omega)\to \mathbb{R}$ as 

\begin{equation}\label{dist1}
\langle(z,Dv),\varphi\rangle:=-\int_\Omega v^*\varphi\operatorname{div}z-\int_\Omega
vz\cdot\nabla\varphi \,\quad \varphi\in C_c^1(\Omega).
\end{equation}

In Anzellotti's theory we need some compatibility conditions, such as $\Div z\in L^1(\Omega)$ and $v\in BV(\Omega)\cap L^\infty(\Omega)$, or $\Div z$ a Radon measure with finite total variation and $v\in BV(\Omega)\cap L^\infty(\Omega)\cap C(\Omega)$. Anzellotti's definition of $( z, Dv)$ can be extended to the case in which $\Div z$ is a Radon measure with finite total variation and $v\in BV(\Omega)\cap L^\infty(\Omega)$ (see \cite[Appendix A]{MST2} and \cite[Section 5]{C}). 
Moreover (\cite{DGS}), if $v\in BV_{\rm{loc}}(\Omega)\cap L^1(\Omega,\operatorname{div}z)$ and $z \in \DM_{\rm{loc}}(\Omega)$,    the distribution defined in \eqref{dist1} is a Radon measure having local finite total variation satisfying 
	\begin{equation*}\label{ec:2}
	|\langle   (z, Dv), \varphi\rangle| \le \|\varphi\|_{L^{\infty}(\omega) }\| z
	\|_{\infty,\omega} \int_{\omega} |Dv|\,,
	\end{equation*}
	for all open set $\omega \subset\subset \Omega$ and for all $\varphi\in
	C_c^1(\omega)$. Here, and throughout the paper,  we use the simplified notation $\|\cdot\|_{q,\omega}$ to indicate the norm of the vector space $L^{q}(\omega,\rn)$; also $\|\cdot\|_{q}:=\|\cdot\|_{q,\Omega}$. 

Moreover  the measure $\vert (z, Dv) \vert$ is absolutely continuous with respect to the measure $\vert Dv \vert$ and, if $v\in BV(\Omega)$, then $(z, Dv)$ has finite total variation.

\medskip
We have the following proposition proved in \cite{DGS}.

\begin{Proposition}\label{nuova}
	Let  $z\in\DM_{\rm{loc}}(\Omega)$ and let $v\in
	BV_{\rm{loc}}(\Omega)\cap L^\infty(\Omega)$. Then $z v\in\DM_{\rm{loc}}(\Omega)$. Moreover the following formula holds in the sense of measures
	\begin{equation*}
	\label{ALPHA}
	\Div(z v)=(\Div z) v^*+(z,Dv).
	\end{equation*}
\end{Proposition}
We observe that, since for every \(v\in BV_{\rm{loc}}(\Omega)\) the measure $(z, Dv)$ is absolutely continuous with respect to $|Dv|$,
it holds
\begin{equation*}\label{theta}
(z, Dv) = \theta(z,Dv,x) \, {|Dv|},
\end{equation*}
where \(\theta(z, Dv, \cdot)\) denotes
the Radon-Nikodym derivative of $(z, Dv)$ with respect to $|Dv|$.
 By  Proposition 4.5 (iii)  of \cite{CDC}, 
for every $z\in \mathcal{D}\mathcal{M}^\infty_{\rm{loc}}(\Omega)$ and \(u\in BV_{\rm{loc}}(\Omega)\), we have
\begin{equation*}
\theta(z,D\Lambda(u),x)=\theta(z,Du,x),
\qquad \text{for \(|D\Lambda(u)|\)-a.e.}\ x\in\Omega,
\end{equation*}
where $\Lambda:\R\to\R$ is a non-decreasing locally Lipschitz function.
We remark that, if $\Lambda$ is an increasing function, then
\begin{equation}\label{radonnik}
\theta(z,D\Lambda(u),x)=\theta(z,Du,x),
\qquad \text{for \(|Du|\)-a.e.}\ x\in\Omega;
\end{equation}
this fact  was already noticed in \cite{LS, GMP} under some additional assumptions.  
Formula \eqref{radonnik} follows by the Chain rule formula (see \cite[Theorem 3.99]{AFP}) observing that 
$$
D^d \Lambda(u) = \Lambda^\prime(\widetilde u) D^d u,
$$
and, as $\Lambda(u)^\pm = \Lambda(u^\pm)$
$$ 
D^j \Lambda(u) = (\Lambda(u^+) - \Lambda(u^-))\, \nu \mathcal H^{N-1}\res J_u\,.
$$

\medskip

The outward normal unit vector $\nu(x)$ is defined for $\mathcal H^{N-1}$-almost every $x\in\partial\Omega$. It follows from Anzellotti's theory that every $z \in \mathcal{DM}^{\infty}(\Omega)$ has
a weak trace on $\partial \Omega$ of the
normal component of  $z$ which is denoted by
$[z, \nu]$. Moreover, it satisfies
\begin{equation}\label{des1}
\|[z,\nu]\|_{L^\infty(\partial\Omega)}\le \|z\|_{\infty}\,.
\end{equation}
We explicitly point out that if $z \in \mathcal{DM}^{\infty}(\Omega)$ and $v\in BV(\Omega)\cap L^\infty(\Omega)$, then
\begin{equation}\label{des2}
v[z,\nu]=[vz,\nu]
\end{equation}
holds (see \cite[Lemma 5.6]{C} or \cite[Proposition 2]{ADS}).

The following generalized  Green type  formula for merely local vector fields $z$ also holds true (\cite{DGS}).

\begin{Proposition}\label{poiu}
	Let $z \in \mathcal{DM}_{\rm{loc}}^{\infty}(\Omega)$ and set $\mu=\Div z$. Let $v\in BV(\Omega)\cap L^\infty(\Omega)$ be such that $v^*\in L^1(\Omega,\mu)$.
	Then $vz\in \mathcal{DM}^{\infty}(\Omega)$ and the following  holds: 
	\begin{equation}\label{GreenIII}
	\int_{\Omega} v^* \, d\mu + \int_{\Omega} (z, Dv) =
	\int_{\partial \Omega} [vz, \nu] \ d\mathcal H^{N-1}\,.
	\end{equation}
\end{Proposition}

Analogously to \eqref{des1}, it can be proved that, for $z\in \DM_{\rm{loc}}(\Omega)$ such that the product $vz\in \DM(\Omega)$ for some $v\in BV(\Omega)\cap L^\infty(\Omega)$,
$$|[vz,\nu]|\le |v_{\res {\partial\Omega}}|\,\|z\|_{\infty}\,\quad\mathcal H^{N-1}\hbox{-a.e. on }\partial\Omega\,.$$

\medskip 

We will finally use the symbol   $\mathcal{S}_p$ to denote the best constant in the Sobolev inequality ($1\leq p< N$), that  is
$$||v||_{L^{p^*}(\Omega)} \le \mathcal{S}_p ||v||_{W^{1,p}_0(\Omega)}, \ \ \forall v \in W^{1,p}_0(\Omega),$$
where $p^*=\frac{Np}{N-p}$. Recall   that 
$$\displaystyle \lim_{p\to 1^{+}} \mathcal{S}_p = \mathcal{S}_1=(N\omega_{N}^{\frac{1}{N}})^{-1}\,,$$
where   	$\omega_{N}$ is the volume of the unit sphere of $\rn$ (see for instance \cite{talenti}).

\subsection{Notations}

 In our arguments we will use several truncating functions. In particular,  for a fixed $k>0$,  we introduce $T_{k}$ and $G_{k}$ as the function defined by
$$
T_k(s)=\max (-k,\min (s,k)), 
$$
and
$$
G_k(s)=(|s|-k)^+ \operatorname{sign}(s). 
$$
Note in particular that   $T_k(s) + G_k(s)=s$, for any $s\in \mathbb{R}$. 
\medskip

If no otherwise specified, we denote by $C$ several constants whose value may change from line to line and, sometimes, on the same line. These values only depend on the data but they do not depend on the indexes of the sequences. We underline the use of the standard convention to do not relabel an extracted compact subsequence.

\section{Main assumptions and core results}
\setcounter{equation}{0}
A primary  aim of this paper is to deal with the following problem
 \begin{equation} \label{pb1lap}
	 \begin{cases}
	 \displaystyle -\Delta_{1} u = h(u)f &  \text{in}\, \Omega, \\
	 u=0 & \text{on}\ \partial \Omega,			
	 \end{cases}
\end{equation}
where $f\in L^{N}(\Omega)$ is positive and $ \Delta_{1} u=\operatorname{div}\left(\frac{D u}{|D u|}\right)$ is the $1$-Laplace operator. As we already mentioned, the general case of a nonnegative $f\in L^{N,\infty}(\Omega)$ will be studied in Sections \ref{sec:fnonnegativa} and \ref{lore}.

On the nonlinearity $h:[0,\infty)\to [0,\infty]$ we assume that it is continuous and finite outside the origin, $h(0)\not =0$, \begin{equation}\label{h1}\tag{h1}
	 	 \displaystyle \exists\;{c_1},\gamma,k_0>0\;\ \text{such that}\;\  h(s)\le \frac{c_1}{s^\gamma} \ \ \text{if} \ \ s\leq k_0,
\end{equation}
 and
 \begin{equation}\label{h2}\tag{h2}
		 \lim_{s\to \infty} h(s):=h(\infty)<\infty\,.
\end{equation}
Let us note that the function $h$ is not necessarily blowing up at the origin so that a bounded continuous function is an admissible choice. 	 
\\We start providing the definition of solution to problem \eqref{pb1lap}.
\begin{defin}\label{weakdef}
 Let $0<f\in L^N(\Omega)$ then a function $u\in BV_{\rm{loc}}(\Omega)\cap L^\infty(\Omega)$ is a solution to problem \eqref{pb1lap} if there exists $z\in \mathcal{D}\mathcal{M}^\infty(\Omega)$ with $||z||_{\infty}\le 1$ such that
	\begin{align}
	& h(u)f\in L^{1}_{\rm{loc}}(\Omega), \label{def_l1loc} 
	\\
	&-\operatorname{div}z =h(u)f \ \ \ \text{in  } \mathcal{D'}(\Omega), \label{def_eqdistr}
	\\
	&(z,Du)=|Du| \label{def_campo} \ \ \ \ \text{as measures in } \Omega,
	\\
	& \text{and one of the following conditions holds:} \notag \\
	&\lim_{\epsilon\to 0}  \fint_{\Omega\cap B(x,\epsilon)} u (y) dy = 0 \ \ \ \text{or} \ \ \ [z,\nu] (x)= -1 \label{def_bordo}\ \ \ \text{for  $\mathcal{H}^{N-1}$-a.e. } x \in \partial\Omega.				\end{align}	
\end{defin}	
\begin{remark} 
Notice that Definition \ref{weakdef} does not depend explicitly on the parameter $\gamma$ in \eqref{h1},  this fact suggesting that an extension to more general  nonlinearities is allowed (see Section \ref{more} below). 
Furthermore, condition \eqref{def_bordo} is a way to give meaning to the homogeneous Dirichlet boundary datum. It is well known that $BV$ solutions to problems involving the $1$-Laplace operator do not necessarily assume the boundary datum pointwise.  A standard weaker request in this framework is that a solution $u\in BV(\Omega)$ satisfies
\begin{equation}\label{bordodcgs}		
	u(1+[z,\nu])(x)= 0 \ \ \ \text{for  $\mathcal{H}^{N-1}$-a.e. } x \in \partial\Omega\,,
\end{equation}
i.e., for $\mathcal{H}^{N-1}$-a.e. $x\in \partial\Omega$, it holds that $u(x)=0$ or $[z,\nu](x)=-1$ (see for instance \cite{MST1,DGS}).  It is not clear whether problem \eqref{pb1lap}  admits, in general, finite energy solutions, that is solutions that are  $BV$ up to the boundary of $\Omega$.  This fact leads to impose \eqref{def_bordo} which is a weaker assumption than \eqref{bordodcgs}  for nonnegative functions (see for instance \cite[Theorem $3.87$]{AFP}). A similar  argument was already exploited  in \cite{OP} when dealing with infinite energy solutions to similar problems involving laplacian type operators. We refer to Section \ref{energiafinita} for further instances of  solutions belonging to the natural energy space $BV(\Omega)$ and how this fact can be related to the smoothness of the domain $\Omega$. 
\end{remark}

Here is  our main existence result in the case of a positive datum $f$:  		
\begin{theorem}\label{existence}
	Let $0<f\in L^{N}(\Omega)$ such that $\displaystyle ||f||_{L^N(\Omega)}< \frac{1}{\mathcal{S}_1 h(\infty)}$, where $h$  satisfies \eqref{h1} and \eqref{h2}. Then there exists a solution $u$ to problem \eqref{pb1lap} in the sense of Definition \ref{weakdef}.
\end{theorem}
In the following result we collect some further qualitative properties enjoyed by the solution we found in Theorem \ref{existence}.  Here and below, in order to unify the presentation, for $\gamma>0$ we set 
\begin{equation}\label{sigma}
\sigma:=\max(1,\gamma)\,.
\end{equation}

\begin{theorem}\label{regularity}
	Let $0<f\in L^{N}(\Omega)$ such that $\displaystyle ||f||_{L^N(\Omega)}< \frac{1}{\mathcal{S}_1 h(\infty)}$, where $h$  satisfies \eqref{h1} and \eqref{h2}. Then the solution $u$ to problem \eqref{pb1lap} found in Theorem \ref{existence} is such that
\begin{equation}\label{eqregul}
		u^{\sigma} \in BV(\Omega). 
\end{equation}
Moreover $\operatorname{div}z \in L^1(\Omega)$ and it holds 
\begin{equation}\label{formtestestese}
	-\int_{\Omega}v \operatorname{div}z = \int_\Omega h(u)fv, \ \ \forall v \in BV(\Omega)\cap L^\infty(\Omega).
\end{equation}
\bk
Finally if $h(0)= \infty$  then $u>0$ a.e. in $\Omega$ while, if $h\in L^\infty([0,\infty))$ and  $||f||_{L^N(\Omega)}< (\mathcal{S}_1 ||h||_{L^\infty([0,\infty))})^{-1}$,  then $u\equiv 0$ a.e. in $\Omega$.
\end{theorem}

If  $h$ is decreasing, then the solutions  obtained  in the previous theorems are unique in the sense specified  by the following theorem.

\begin{theorem}\label{uniqueness}
Let $h$ be a decreasing function. Then, under the assumptions of Theorem \ref{existence},  there is only one solution $u$ to problem  \eqref{pb1lap} in the sense of Definition \ref{weakdef} such that $u^{\sigma}\in BV(\Omega)$.			
\end{theorem}
\begin{remark}\label{bv} We stress  that Theorem \ref{existence} is new also in the case of a bounded nonlinearity $h$. 
A  remark is in order concerning \eqref{eqregul}. 
First of all, if $\gamma\leq 1$ then \eqref{eqregul} says that solutions always belong to the natural space of functions with finite  $BV$ norm, accordingly  with the case $p>1$ where solutions with finite $W^{1,p}_0$\-energy are always achieved when the datum $f$ belongs to  $L^{(p^*)'}(\Omega)$ (\cite{BO, CST, DCA, GMM2}). 
If $\gamma>1$ then \eqref{eqregul} is a sort of counterpart, as $p$ tends to $1^{+}$,  of the regularity for the solutions to the Dirichlet problem associated to $-\Delta_{p} v = fv^{-\gamma}$. In \cite{BO, DCA, GMM} it is shown that, although $v$ has no finite energy,  in general  $v^{\frac{\gamma-1+p}{p}}$ does, and this  also gives a  (weak) sense to the  boundary datum.   Again, we refer to Section \ref{finite} for more comments  on this and on how the belonging of $u$ to  $BV(\Omega)$ can depend on both the degeneracy of the datum $f$ and  the geometry  of $\partial\Omega$. 

\medskip 

We finally emphasize the regularizing effect given by the, possibly singular, nonlinear term $h$. If $h(0)= \infty$ then $u>0$ a.e. in $\Omega$;  this is a striking difference with the bounded case; as a consequence of a result in \cite{CT}, in fact, if $h(s)\equiv 1$ and  $\displaystyle ||f||_{L^N(\Omega)}< {\mathcal{S}_1}^{-1}$, then $u=0$. The extension of this property to the case of a general bounded nonlinearity is given by the last assertion of Theorem \ref{regularity}.

As expected, also the behavior at infinity of the nonlinearity $h$ plays a role; in fact, if $h(\infty)=0$ then no smallness assumptions need to be imposed on the data in contrast with the {\it linear} right hand side case (\cite{CT,K,KS,MST1}).

\end{remark}
\section{Existence of a solution for $p>1$ and for a nonnegative $f\in L^{(p^*)'}(\Omega)$}\label{carrie}
 Here we set the theory in the case of a $p$-laplace principal part for a fixed $1<p<N$.  This case  represents the basis of  the approximation scheme we will use to prove Theorem \ref{existence}. Existence of solutions (and  uniqueness when expected) for this kind of problems has its own interest and we will present it in full generality.  Let us consider
\begin{equation}\label{pbplap}
	\begin{cases}
	\displaystyle - \Delta_p u_p= h(u_p)f &  \text{in}\, \Omega, \\
	u_p=0 & \text{on}\ \partial \Omega,
	\end{cases}
\end{equation}	
where $f\in L^{(p^*)'}(\Omega)$ is a nonnegative function and $h$ satisfies both \eqref{h1} and \eqref{h2}. We precise the notion of solution to problem  \eqref{pbplap} we adopt.
\begin{defin}\label{distributional}
	A nonnegative function $u_p\in W^{1,p}_{\rm{loc}}(\Omega)$ is a \emph{distributional solution} to problem \eqref{pbplap} if 

\begin{equation}
	h(u_p)f\in L^{1}_{\rm{loc}}(\Omega), \label{pl1loc}
\end{equation}

	\begin{equation}
	 G_{k}(u_{p})\in W^{1,p}_{0}(\Omega) \ \ \ \text{for all}\ k>0\,,\label{pbordodef}
	\end{equation}
and
	\begin{equation}\int_{\Omega}|\nabla u_p|^{p-2} \nabla u_p\cdot \nabla \varphi = \int_{\Omega}h(u_p)f\varphi, \label{pweakdef}
	\end{equation}
for every $\varphi \in C^{1}_{c}(\Omega)$.	
\end{defin}

\begin{remark}
 Notice that condition \eqref{pbordodef} is the same used in \cite{CST} for $p>{1}$ in the model case $h(s)=s^{-\gamma}$ in order to give sense to the boundary datum and to ensure uniqueness under suitable assumptions on both the datum and/or the domain. As a matter of fact, if $h$ is non-increasing, then a straightforward re-adaptation of the proof of Theorem $1.5$ in \cite{CST} shows that the same uniqueness property holds for problem  \eqref{pbplap}. In particular, if $\Omega$ star-shaped, one can show that uniqueness of distributional solutions in the sense of Definition \ref{distributional} holds if $f\in L^1(\Omega)$, while some regularity on $f$ is needed if $\gamma>1$ and the domain is more general.

	As we already mentioned, an alternative  way one can weakly intend the boundary datum in problem \eqref{pbplap} is by requesting that
\begin{equation}
u_{p}^{\frac{\sigma-1+p}{p}}\in W^{1,p}_{0}(\Omega)\,. \label{sciunzi} 	
\end{equation}
Property \eqref{sciunzi} is the same given in \cite{DCA},  it  is consistent with the case $p=2$ (\cite{BO}), and, as we said,  it has  its counterpart as $p\to1^{+}$ (i.e. $u^{\sigma}\in BV(\Omega)$).   
Let us observe that, if $f\in L^{N}(\Omega)$, then the solutions we will construct in Theorem \ref{existence_p} below enjoy \eqref{sciunzi} (see \eqref{exbordo} and Section \ref{mrd} below) that, by a  direct computation, can be shown to imply \eqref{pbordodef}. 
\\We finally  want to stress some striking differences with the model case $h(s)=s^{-\gamma}$, $\gamma>0$. In this case, in fact, the behavior of $h$ at infinity also plays a role and \eqref{sciunzi} can be shown to hold for any nonnegative  $f\in L^{1}(\Omega)$ if $\gamma\geq1$. If $\gamma<1$ then solutions in $W^{1,p}_{0}(\Omega)$ exist for a datum $f\in L^{(\frac{p^{\ast}}{1-\gamma})'}(\Omega)$ (see \cite{DCA}).  In this sense, as we do not assume any behavior for $h$ at infinity, our summability assumption on the datum $f$ can be considered to be optimal. See also Theorem \ref{more1} below for further regularity results depending on the summability of the datum $f$. 
	
\end{remark}

The following existence result holds. 
\begin{theorem}\label{existence_p}
Let $0\le f\in L^{(p^*)'}(\Omega)$ and let $h$ satisfy \eqref{h1} and \eqref{h2}. Then there exists a distributional solution $u_p$ to problem \eqref{pbplap}.
\end{theorem}

\subsection{A priori estimates} In order to prove Theorem \ref{existence_p} we need to establish some general a priori estimates that will be the content of this section. 
We look for a priori estimates for a weak solution to the following problem
\begin{equation}\label{prioripbplap}
\begin{cases}
\displaystyle - \Delta_p v= \overline{h}(v)f &  \text{in}\, \Omega, \\
u=0 & \text{on}\ \partial \Omega,
\end{cases}
\end{equation}	
where $1<p<N$, $\overline h:[0,\infty)\to [0,\infty)$  is bounded continuous function satisfying \eqref{h2},  and $f$ is a nonnegative function in $L^{(p^{*})'}(\Omega)$. For a \emph{weak solution} to problem \eqref{prioripbplap} we mean a function $v\in W^{1,p}_0(\Omega)$ such that

\begin{equation}\label{prioriweakdef}
\int_{\Omega}|\nabla v|^{p-2}\nabla v \cdot \nabla \varphi = \int_{\Omega} \overline{h}(v)f\varphi \ \ \ \forall \varphi\in {W^{1,p}_0(\Omega)}.
\end{equation}

We recall  that $\sigma$ is defined by \eqref{sigma},  and,  for any $k>0$,  we set for simplicity
$$
\epsilon_{k}:=\sup_{s\in[k,\infty)} \overline{h}(s)-\overline{h}(\infty).
$$
Observe that  $\epsilon_{k}\geq 0$  and 
$$
\lim_{k\to\infty}\epsilon_{k}=0\,.
$$
 We have the following

\begin{lemma}\label{priorilemma}
	Let $0\le f\in L^{(p^{*})'}(\Omega)$ and let $\overline h$ be a bounded continuous function satisfying \eqref{h2}. Then every weak solution $v$ to problem \eqref{prioripbplap} satisfies
	\begin{equation}\label{apriorigk}
	||G_{k}(v)||_{W^{1,p}_{0}(\Omega)}\le C(p,\mathcal{S}_p, \epsilon_k, ||f||_{L^{(p^{*})'}(\Omega)}) \ \ \ \text{for all}\ k>0\,,
	\end{equation}	
	and 
	\begin{equation}\label{priorieq}
	 ||v||_{W^{1,p}(\omega)}\le C(p,\mathcal{S}_p, \overline{h}(\infty), ||f||_{L^{(p^{*})'}(\Omega)},\omega) \ \ \ \text{for all}\ \omega \subset \subset \Omega.
	\end{equation}
						
\end{lemma}
\begin{proof}
	We first look for \eqref{apriorigk}. 						
	For a fixed $k>0$, one takes $G_k(v)$ as test function in \eqref{prioriweakdef}, and using the H\"older and the Sobolev inequality,  one readily gets
	\begin{equation*}
	\begin{aligned}\label{stimagk}
	\int_{\Omega}|\nabla G_k(v)|^{p} &= \int_{\Omega} \overline{h}(v)fG_k(v) \le (\overline{h}(\infty)+\epsilon_k)||f||_{L^{(p^{*})'}(\Omega)} \left(\int_{\Omega} G_k^{p^*}(v)\right)^\frac{1}{p^*} 
	\\ 
	&\le (\overline{h}(\infty)+\epsilon_k)||f||_{L^{(p^{*})'}(\Omega)} \mathcal{S}_p\left(\int_{\Omega} |\nabla G_k(v)|^p\right)^\frac{1}{p}\,,
	\end{aligned}
	\end{equation*}
	that gives 
	\begin{align}\label{priorigk}
	||G_{k}(v)||_{W^{1,p}_{0}(\Omega)} \le [( \overline{h}(\infty)+\epsilon_{k})||f||_{L^{(p^{*})'}(\Omega)}\mathcal{S}_p]^{\frac{1}{p-1}}\,.
	\end{align}

	Now, in order to obtain \eqref{priorieq}, we look for a local estimate on $T_{k}(v)$.  Consider $\omega\subset\subset \Omega$ and $\phi \in C^1_c(\Omega)$ as a cut-off function for $\omega$, i. e. $0\leq \phi\leq 1$, $\phi=1$ on $\omega$ and $|\nabla \phi |\leq c_{\omega}$, where $c_{\omega}$ is a constant that only depends on $\dist(\omega,\partial \Omega)$.  We take  $T_k(v)\phi^p$ as test function in \eqref{prioriweakdef}, we have
	\begin{align*}
	\int_{\Omega}|\nabla T_k(v)|^{p}\phi^p + p\int_{\Omega}|\nabla v|^{p-2}\nabla v \cdot \nabla \phi \phi^{p-1}T_k(v) = \int_{\Omega} \overline{h}(v)fT_k(v)\phi^p, 
	\end{align*}			
	and so
	\begin{equation}
	\begin{aligned}\label{priori=}
	\int_{\Omega}|\nabla T_k(v)|^{p}\phi^p &\le  pk\left|\int_{\Omega}|\nabla T_k(v)|^{p-2}\nabla T_k(v) \cdot \nabla \phi \phi^{p-1}\right| 
	\\ 
	&+ pk\left|\int_{\Omega}|\nabla G_k(v)|^{p-2}\nabla G_k(v) \cdot \nabla \phi \phi^{p-1}\right| +\int_{\Omega} \overline{h}(v)fT_k(v)\phi^p. 
	\end{aligned}
	\end{equation}	
 By \eqref{prioriweakdef}, choosing $\varphi =\phi^p$, we have
 \begin{equation}
 \begin{aligned}\label{prioritk3}
 \int_{\Omega} \overline{h}(v)fT_k(v)\phi^p &\le k \int_{\Omega} \overline{h}(v)f\phi^p = pk  \int_{\Omega} |\nabla v|^{p-2}\nabla v \cdot \nabla \phi \phi^{p-1}
 \\ 
 &\le pk \left| \int_{\Omega} |\nabla T_k(v)|^{p-2}\nabla T_k(v) \cdot \nabla \phi \phi^{p-1} \right| 
 \\ 
 &+ pk \left| \int_{\Omega} |\nabla G_k(v)|^{p-2}\nabla G_k(v) \cdot \nabla \phi \phi^{p-1} \right|, 
 \end{aligned}
 \end{equation}	
and collecting \eqref{prioritk3} and \eqref{priori=} one deduces
	\begin{equation*}
	\begin{aligned}\label{priori=2}
	\int_{\Omega}|\nabla T_k(v)|^{p}\phi^p &\le  2pk\left|\int_{\Omega}|\nabla T_k(v)|^{p-2}\nabla T_k(v) \cdot \nabla \phi \phi^{p-1}\right| 
	\\ 
	&+ 2pk\left|\int_{\Omega}|\nabla G_k(v)|^{p-2}\nabla G_k(v) \cdot \nabla \phi \phi^{p-1}\right|. 
	\end{aligned}
	\end{equation*}	
	By the Young inequality and by \eqref{priorigk} the previous implies, for a positive $\eta$ to be fixed later, that 
	\begin{equation}
	\begin{aligned}\label{prioritk}
	\int_{\Omega}|\nabla T_k(v)|^{p}\phi^p &\le  (p-1)\eta^{p'}\int_{\Omega}|\nabla T_k(v)|^{p}\phi^{p} + \frac{(2k)^p}{\eta^p} \int_{\Omega}|\nabla\phi|^{p}
	\\ 
	&+ (p-1)\int_{\Omega}|\nabla G_k(v)|^{p}\phi^{p} + (2k)^p \int_{\Omega}|\nabla\phi|^{p}
	\\ 
	&\stackrel{\eqref{priorigk}}{\le} (p-1)\eta^{p'}\int_{\Omega}|\nabla T_k(v)|^{p}\phi^{p} + \frac{(2k)^p}{\eta^p} \int_{\Omega}|\nabla\phi|^{p}
	\\ 
	&+ (p-1) [(\overline{h}(\infty)+\epsilon_{k})||f||_{L^{(p^{*})'}(\Omega)}\mathcal{S}_p ]^{\frac{p}{p-1}}+ (2k)^p \int_{\Omega}|\nabla\phi|^{p}. 
	\end{aligned}	
\end{equation}	
	We fix $\eta$ such that $1-	(p-1)\eta^{p'}=\frac12$,  that is $\eta=\left(\frac{1}{2(p-1)}\right)^{\frac{1}{p'}}$, then \eqref{prioritk} implies	
	
\begin{equation}
	\begin{aligned}\label{prioritk4}
	\int_{\omega}|\nabla T_k(v)|^{p} &\le  2^{p+1} k^p [(2(p-1))^{p-1}+1] \int_{\Omega}|\nabla\phi|^{p} 
	\\ 
	&+ 2(p-1)[(\overline{h}(\infty)+\epsilon_{k})||f||_{L^{(p^{*})'}(\Omega)}\mathcal{S}_p ]^{\frac{p}{p-1}}
	\\
	&\le 2^{p+1} k^p [(2(p-1))^{p-1}+1] c_\omega^{p} |\Omega| 
	\\ 
	&+ 2(p-1)[(\overline{h}(\infty)+\epsilon_{k})||f||_{L^{(p^{*})'}(\Omega)}\mathcal{S}_p ]^{\frac{p}{p-1}}.
	\end{aligned}	
\end{equation}
	Now we fix $k$ large enough in order to have  $\epsilon_{k}\leq 1$ and we collect  \eqref{priorigk} and \eqref{prioritk4} yielding  \eqref{priorieq}. 
\end{proof}\bk

\subsection{Proof of Theorem \ref{existence_p}}

\begin{proof}[Proof of Theorem \ref{existence_p}]
If $h(0)< \infty$ then the existence of a solution belonging to $W^{1,p}_0(\Omega)$ follows by standard application of a fixed point argument,  so that, without loosing generality, we assume $h(0)= \infty$. Let us introduce the following scheme of approximation
\begin{equation}\label{pbplapapprox}
\begin{cases}
\displaystyle - \Delta_p u_{n}= h_n(u_{n})f &  \text{in}\, \Omega, \\
u_{n}=0 & \text{on}\ \partial \Omega,
\end{cases}
\end{equation}	
where $h_n(s)=T_n(h(s))$, for $s\in[0,\infty)$,  and  $n > h(\infty)$. The existence of such $u_n\in W^{1,p}_0(\Omega)$ follows again by standard  Schauder fixed point theorem. Moreover, $u_{n}$ is easily seen to be nonnegative. 
We  apply Lemma \ref{priorilemma} to $u_n$, deducing 
\begin{equation}\label{also}
||u_n||_{W^{1,p}(\omega)}\le C_\omega, \ \ \ \forall \omega \subset\subset \Omega
\end{equation}
which implies that, up to subsequences,  $u_n$ weakly converges in $W^{1,p}(\omega)$ and a.e. in $\Omega$ to a function $u_p$.
Moreover,  by weak lower semicontinuity in \eqref{apriorigk} (applied to $u_{n}$)  one also gets the boundary condition \eqref{pbordodef}.
\medskip 
Now we  prove \eqref{pweakdef}.  First of all, using \eqref{pbplapapprox} and  \eqref{also} we observe that  $h_n(u_n)f$ is locally bounded in $L^1(\Omega)$;  in fact, for any  $\varphi\in C^1_c(\Omega)$ one  obtains
\begin{equation*}\label{fatoul1}
\int_{\Omega} h_n(u_n)f \varphi = \int_{\Omega} |\nabla u_n|^{p-2}\nabla u_n \cdot \nabla \varphi \le \frac{1}{p'}\int_{supp \varphi} |\nabla u_n|^{p} + \frac{1}{p} \int_{\Omega} |\nabla \varphi|^p \le C \,,
\end{equation*}
and $\varphi$ can be chosen to be a cut-off function for any compact subset $\omega$ of $\Omega$.  We can then  apply Theorem $2.1$ of \cite{bm} in order to deduce that $\nabla u_n$ converges a.e. in $\Omega$ to $\nabla u_p$. In particular,  $|\nabla u_n|^{p-2}\nabla u_n$ locally strongly converges to   $|\nabla u_p|^{p-2}\nabla u_p$ in $L^{q}(\Omega, \mathbb{R}^{N})$, for any $q<p'$. 		
Moreover,  by  the Fatou lemma,  it follows
\begin{align}\label{limitel1}
\int_{\Omega}h(u_p)f \varphi \le C,
\end{align}			
for any nonnegative $\varphi\in C^1_c(\Omega)$, which implies  (recall we are assuming $h(0)=\infty$) 
\begin{equation}\label{u0f0}
\{u_p=0\}\subset \{f=0\},
\end{equation}
 up to a set of zero Lebesgue measure.
\medskip 				
Now we consider  $V_\delta(u_n)\varphi$,  as a test for the weak formulation of \eqref{pbplapapprox} where 
\begin{align}\label{V_delta}
\displaystyle
V_\delta(s):=
\begin{cases}
1 \ \ &s\le \delta, \\
\displaystyle\frac{2\delta-s}{\delta} \ \ &\delta <s< 2\delta, \\
0 \ \ &s\ge 2\delta,
\end{cases}
\end{align}
 and $0\leq \varphi\in C^1_c(\Omega)$.  Observing   that $V'_\delta(s)\le 0$ one deduces that 
\begin{align*}
\int_{\Omega} h_n(u_n)f V_\delta(u_n)\varphi \le \int_{\Omega} |\nabla u_n|^{p-2}\nabla u_n \cdot \nabla \varphi V_\delta(u_n)\,.
\end{align*}		
We first pass to the limit with respect to $n$. Using the strong convergence of $|\nabla u_n|^{p-2}\nabla u_n$ and the a.e. and $\ast$-weak convergence of $V_\delta(u_n)$ in $L^{\infty}(\Omega)$,  we can pass to the limit on the right hand side.  By Fatou's lemma on the left hand side we then deduce
\begin{align*}
\int_{\{u_p\le \delta\}} h(u_p)f \varphi \le \int_{\Omega} |\nabla u_p|^{p-2}\nabla u_p \cdot \nabla \varphi V_\delta(u_p).
\end{align*}			
Now we take in the previous $\delta \to 0^+$ obtaining, by $\ast$-weak convergence
\begin{equation*}\label{alre}
\lim_{\delta\to 0^+}\int_{\{u_p\le \delta\}} h(u_p)f \varphi \le \int_{\{u_p=0\}} |\nabla u_p|^{p-2}\nabla u_p \cdot \nabla \varphi=0.
\end{equation*}		
Therefore,  
\begin{align*}
\int_{\Omega}h_n(u_n)f\varphi=  \int_{\{u_n> \delta\}}h_n(u_n)f\varphi +\epsilon(n,\delta)
\,,
\end{align*}	
where $\epsilon(n,\delta)$ is a quantity that vanishes as first $n$ goes to $\infty$ and then $\delta$ goes to zero. 
We observe that, without loss of generality, we  can always assume that  $\delta \notin \{\eta: |\{u_p=\eta\}|>0\}$ which is at  most a countable set;  this will imply, in particular,  that  $\chi_{\{u_n> \delta\}}$ converges a.e. in $\Omega$	to $\chi_{\{u_p> \delta\}}$ as $n\to\infty$.  Moreover,  we both  have  
$$h_n(u_n)f\chi_{\{u_n> \delta\}}\varphi \le \sup_{s\in[\delta, \infty)} h(s) \ f \varphi \in L^1(\Omega),$$
and  
$$h(u_p)f\chi_{\{u_p> \delta\}}\varphi \le h(u_p)f\varphi \stackrel{\eqref{limitel1}}{\in} L^1(\Omega)\,.$$
We can then  apply the Lebesgue theorem in order to deduce that 
$$\lim_{\delta\to 0^+}\lim_{n\to \infty} \int_{\Omega}h_n(u_n)f \varphi = \int_{\{u_p> 0\}}h(u_p)f\varphi \stackrel{\eqref{u0f0}}{=} \int_{\Omega}h(u_p)f\varphi\,.$$	
On  the left hand side of the weak formulation of \eqref{pbplapapprox} we pass to the limit using the  weak convergence and a.e. convergence of $\nabla u_n$ to $\nabla u_p$, finally obtaining 
\begin{equation*}\label{pdistrsol}
\int_{\Omega}|\nabla u_p|^{p-2} \nabla u_p\cdot \nabla \varphi = \int_{\Omega}h(u_p)f\varphi.
\end{equation*} 		
for every nonnegative $\varphi \in C^1_c(\Omega)$ from which easily  \eqref{pweakdef} follows.
\end{proof}
\begin{remark}\label{remtest}
An important remark is that,  a careful  re-adaptation of  the proof of Theorem \ref{existence_p}, can  allow to  slightly improve  the set of admissible test function in \eqref{pweakdef}. Precisely, following the same steps, one can actually realize  that \eqref{pweakdef} holds true for $\varphi \in W^{1,p}(\Omega)$ having compact support in $\Omega$. We will use this property later once we will pass to the limit as $p\to 1^{+}$.
	\bk
\end{remark}
	\begin{remark} 
			We also highlight the fact that if $h(\tilde{k})=0$ for some $\tilde{k}>0$ then we can retrieve some more informations on $u_p$. Indeed Theorem \ref{existence_p} guarantees the existence of a distributional solution to problem 
				\begin{equation*}
				\begin{cases}
				\displaystyle - \Delta_p u_p= \underline{h}(u_p)f &  \text{in}\, \ \Omega, \\
				u_p=0 & \text{on}\ \partial \Omega,
				\end{cases}
				\end{equation*}			
			where 
				\begin{align*}
				\displaystyle
				\underline{h}(s):= 
				\begin{cases}
					h(s) \ \ &\text{if } s\le \tilde{k},
					\\
					0 \ \ &\text{if}\ s> \tilde{k}.
				\end{cases}
				\end{align*}			
			As far as  it has been obtained, $u_p$ is the almost everywhere limit of the $u_n$ solutions of \eqref{pbplapapprox} with $\underline{h}_n$ in place of $h_n$. By considering  $G_{\tilde{k}}(u_n)$ as a test function in \eqref{pbplapapprox} one thus obtain \begin{equation*}
					\int_\Omega |\nabla G_{\tilde{k}}(u_n)|^p \le \int_{\Omega} \underline{h}_n(u_n)f G_{\tilde{k}}(u_n) = 0, 
				\end{equation*}				
			that implies 
				\begin{equation*}
					||u_n||_{L^\infty(\Omega)}\le \tilde{k}, 
				\end{equation*}			
			and,  so 				
				\begin{equation}\label{uplimitata}
					||u_p||_{L^\infty(\Omega)}\le \tilde{k}. 
				\end{equation}						
			Since  $u_p$ satisfies
				\begin{equation*}
					\int_{\Omega}|\nabla u_p|^{p-2} \nabla u_p\cdot \nabla \varphi = \int_{\Omega}\underline{h}(u_p)f\varphi,
				\end{equation*}
			for every $\varphi \in C^{1}_{c}(\Omega)$, using  \eqref{uplimitata},   one also has
				\begin{equation*}
				\int_{\Omega}|\nabla u_p|^{p-2} \nabla u_p\cdot \nabla \varphi = \int_{\Omega}h(u_p)f\varphi\,,
				\end{equation*}				
			namely $u_p$ solves 
\begin{equation*}
				\begin{cases}
				\displaystyle - \Delta_p u_p= h(u_p)f &  \text{in}\, \Omega, \\
				u_p=0 & \text{on}\ \partial \Omega\,,
				\end{cases}
				\end{equation*}
				as it also satisfies \eqref{pbordodef}. A trivial observation is that, in this case, $h$ needs not to be bounded at infinity. 	
					
		\end{remark}

 \subsection{More regular data  $f\in L^{m}(\Omega)$ with $(p^{\ast})' \leq  m <N$} \label{mrd}
 
  As we already observed, a natural question concerning solutions $u_{p}$ to problem \eqref{pbplap} is whether they enjoy property \eqref{sciunzi}, as this is the case, for instance, in the model $h(s)=s^{-\gamma}$ (even for merely integrable data). As far as our nonlinear term is concerned the behavior at infinity of $h$ plays a crucial role so that \eqref{sciunzi} is not expected in general  for such a large class of data. 
What is true in general, for $f\in L^{(p^{\ast})'}(\Omega)$, is that the solutions  $u_{p}$ of \eqref{pbplap} found in Theorem \ref{existence_p} satisfy
\begin{equation}T_k^{\frac{\sigma-1+p}{p}}(u_{p}) \in W^{1,p}_0(\Omega) \ \ \ \text{and} \ \ \ G_k(u_{p})\in W^{1,p}_0(\Omega) \ \ \ \text{for all } k>0.\label{regGk}\end{equation}
The first in \eqref{regGk} can be easily obtained by taking   $T_k^\sigma(u_n)$ as a test function in \eqref{pbplapapprox} while the second one has been already shown to follow from  \eqref{apriorigk}. 

However, if the datum $f$ is more regular something more can be said 
\begin{theorem}\label{more1}
	Let $0 \le f\in L^m(\Omega)$ with $m$ such that $(p^*)' \leq m< \frac{N}{p}$ and let $k>0$.  Then there exists $C>0$ such that  the solution $u_{p}$ to \eqref{pbplap} found in Theorem \ref{existence_p} satisfies
	\begin{equation}\label{stimagkreg}
	\|G_k^{\frac{q-1+p}{p}}(u_{p})\|_{W^{1,p}_{0}(\Omega)} \le C,
	\end{equation}
	for any $q$ such that $\displaystyle 1\le q\le \frac{p^{*}(p-1)}{m'p-p^{*}}$. Moreover, if $f\in L^m(\Omega)$ with $m= \frac{pN-N+N\sigma}{pN-N+p\sigma}$, then 
		\begin{equation}\label{stimaup}
	\|u_{p}^{\frac{\sigma-1+p}{p}}\|_{W^{1,p}_{0}(\Omega)} \le C\,.
	\end{equation}
	
	Finally, if  $f\in L^m(\Omega)$ with $m> \frac{N}{p}$ then $u_{p}\in L^{\infty}(\Omega)$. 
\end{theorem}
\begin{proof} 
For a fixed $k>0$ let us define the following auxiliary function
$$
\Phi(t):=
\begin{cases}
	t^\sigma \ \ &\text{if}\ \ t\le k,\\
	(t-k)^q +k^\sigma \ \ &\text{if}\ \ t> k,	
\end{cases}
$$
and  take $\Phi(u_n)$ as test in  \eqref{pbplapapprox} obtaining, using the H\"older inequality and the assumption on $q$, that
\begin{equation}
\begin{aligned}\label{stimareg}
	&\left(\frac{p}{\sigma-1+p}\right)^p\sigma\int_{\Omega} |\nabla T_k^{\frac{\sigma-1+p}{p}}(u_n)|^p + \left(\frac{p}{q-1+p}\right)^p q\int_{\Omega} |\nabla G_k^{\frac{q-1+p}{p}}(u_n)|^p \\ 
	& \le \max_{s\in [0,k]}h(s)s^\sigma \int_{\{u_n \le k\}}f
	+ \sup_{s\in[k,\infty)} h(s)\int_{\{u_n > k\}} f G_k^q(u_n) + \sup_{s\in[k,\infty)} h(s)k^\sigma \int_{\Omega}f \\ & \le   C + \sup_{s\in[k,\infty)} h(s) ||f||_{L^m(\Omega)}|\Omega|^{\frac{1}{m'} - \frac{qp}{(q-1+p)p^*}} \left(\int_{\Omega}G_k^{\frac{(q-1+p)p^*}{p}}(u_n)\right)^{\frac{qp}{(q-1+p)p^*}}\,.  
\end{aligned}
\end{equation}
From \eqref{stimareg},  using the Sobolev inequality we deduce 
\begin{align*}
& \left(\frac{p}{q-1+p}\right)^p q \mathcal{S}_p^{-p}\left(\int_{\Omega} G_k^{\frac{(q-1+p)p^*}{p}}(u_n)\right)^\frac{p}{p^*} \le \left(\frac{p}{q-1+p}\right)^p q\int_{\Omega} |\nabla G_k^{\frac{q-1+p}{p}}(u_n)|^p  
\\&\le  C + \sup_{s\in[k,\infty)} h(s) ||f||_{L^m(\Omega)}|\Omega|^{\frac{1}{m'} - \frac{qp}{(q-1+p)p^*}} \left(\int_{\Omega}G_k^{\frac{(q-1+p)p^*}{p}}(u_n)\right)^{\frac{qp}{(q-1+p)p^*}} 
\\ &{\le} C + \varepsilon\sup_{s\in[k,\infty)} h(s) ||f||_{L^m(\Omega)}|\Omega|^{\frac{1}{m'} - \frac{qp}{(q-1+p)p^*}} \left(\int_{\Omega}G_k^{\frac{(q-1+p)p^*}{p}}(u_n)\right)^{\frac{p}{p^*}} 
\\&+ C_\varepsilon\sup_{s\in[k,\infty)} h(s) ||f||_{L^m(\Omega)}|\Omega|^{\frac{1}{m'} - \frac{qp}{(q-1+p)p^*}},
\end{align*} 
where in the last step we also used Young's inequality. Up to suitably choose $\varepsilon$, this gives that $G_k(u_n)$ is bounded in $L^\frac{(q-1+p)p^*}{p}(\Omega)$. Hence by \eqref{stimareg} one gets
\begin{align}\label{oneg}
&\int_{\Omega} |\nabla T_k^{\frac{\sigma-1+p}{p}}(u_n)|^p + \int_{\Omega} |\nabla G_k^{\frac{q-1+p}{p}}(u_n)|^p \le C\,,
\end{align}
that implies \eqref{stimagkreg} by weak lower semicontinuity.

\medskip 

If $f\in L^m(\Omega)$ with $m= \frac{pN-N+N\sigma}{pN-N+p\sigma}$, then $q=\sigma$ and one obtains \eqref{stimaup} by weak lower semicontinuity in \eqref{oneg} and considering, for instance,  $k=1$ (recall $T_{1}(s)+G_{1}(s)=s$). 
\medskip 

Now, let $f\in L^m(\Omega)$ with $m> \frac{N}{p}$. It suffices to observe that in  \eqref{pbplapapprox} it is not restrictive to choose $\displaystyle  n>\sup_{s\in [k_{0},\infty)} h(s)$, that is one can possibly  truncate only near the singularity. Hence, if we multiply \eqref{pbplapapprox} by $G_{k}(u_{n})$ one readily has
$$
\into |\nabla G_{k}(u_{n})|^{p}\leq \sup_{s\in[k_{0},\infty)} h(s) \into f G_{k}(u_{n})\,,
$$
and one can apply standard Stampacchia's method in order to get the boundedness of $u_{p}$. 
\end{proof}
\begin{remark}
	First of all observe that, if $\gamma\leq 1$ (i.e. $\sigma=1$),  then  \eqref{stimaup}  holds for $m=(p^{\ast})'$ and one has that the solutions are globally $W^{1,p}_{0}(\Omega)$ no matter of the behavior of $h$ at infinity.  
	
	What actually holds for any $\gamma>0$, is that   $q=\frac{p^{*}(p-1)}{m'p-p^{*}}= 1$ if $m=(p^*)'$ and so \eqref{stimagkreg} is in  continuity with  the result of Theorem \ref{existence_p}  (i.e. $G_k(u_p) \in W^{1,p}_0(\Omega)$).

	Also notice that $q\to \infty$ as $m \to \frac{N}{p}$, formally implying that every power of $G_k(u_p)$ stands in $W^{1,p}_0(\Omega)$. Observe that $ \frac{pN-N+N\sigma}{pN-N+p\sigma}<\frac{N}{p}< N$ for any $p>1$. In the following section we will consider data in $L^{N}(\Omega)$; hence,  we shall be backed  to  look for (bounded, in fact) solutions satisfying  $u_{p}^\frac{\sigma-1+p}{p}\in W^{1,p}_0(\Omega)$ (see \eqref{exbordo} below). 
	
\end{remark}

\subsection{Uniform estimates in the case $f\in L^{N}(\Omega)$}\label{4.4}

In this section we are going to prepare the proofs of Theorems \ref{existence}, \ref{regularity}, and \ref{uniqueness}.  As one would like to let $p\to 1^{+}$ we will need some \emph{uniform estimates} with respect to $p$.  In order to do that we introduce a family of test functions that behaves differently as $u_{p}\sim 0$ and $u_{p}\sim\infty$. 
For fixed $p>1$, we define the following auxiliary function
 	\begin{equation}\label{psi}
 	\psi_p(s):= 
 	\begin{cases}
 	s^{\sigma} k_0^{-\frac{(\sigma-1)(p-1)}{p}} \ \ \ &\text{if }s\le k_0, \\
 	s^{\frac{\sigma-1+p}{p}} &\text{if }s>k_0.
 	\end{cases}
 	\end{equation}
	Observe that $$\psi_p(s)\to s^{\sigma},\ \  \text{as}\ \ \  p\to 1^{+}.$$ 
 We also define $\Gamma_p (s) := \int_{0}^{s} \psi'^{\frac{1}{p}}_p(t) dt.$ For the sake of exposition, here and below we will tacitly understand that $N>1$. Actually, in the case $N=1$
 many straightforward simplifications will appear in the arguments where, with a little abuse of notation,  one can mostly  think of $\frac{N}{N-1}:=\infty$. 
 
 \medskip
 
 We have the following
\begin{lemma}\label{lemmapfisso}
	Let $h$ satisfy \eqref{h1} and \eqref{h2}, $0\le f\in L^N(\Omega)$ such that  $\dys||f||_{L^{N}(\Omega)}< \frac{1}{\mathcal{S}_1 h(\infty)}$,  and let  $u_p$ be the solution to \eqref{pbplap} found in Theorem \ref{existence_p}. Then   there exists $p_0>1$ such that for any  $p\in (1,p_0)$ one has  
	\begin{equation}\label{exstima}
		 ||u_p||_{W^{1,p}(\omega)}\le C(\mathcal{S}_1, h(\infty), ||f||_{L^{N}(\Omega)}, \omega) \ \ \ \text{for all }\omega\subset\subset\Omega\bk\,,
	\end{equation}
	and	
	\begin{equation}\label{exbordo}
		||u_p^{\frac{\sigma-1+p}{p}}||_{W^{1,p}_0(\Omega)} \le C(\mathcal{S}_1, \sup_{s\in[k_{0},\infty)}h(s), ||f||_{L^{N}(\Omega)},c_1,  |\Omega|).
	\end{equation}
Moreover, there exists $\overline k>0$ such that, for any $k\geq \overline k$
	\begin{equation}\label{exubounded}
		||G_k(u_p)||_{L^{\frac{N}{N-1}}(\Omega)} \le \frac{p-1}{p} |\Omega|C(\mathcal{S}_1, h(\infty), ||f||_{L^{N}(\Omega)}, \overline k) \,.
	\end{equation}

	Finally,  
	\begin{equation}\label{dispotenza}
		\int_{\Omega} |\nabla \Gamma_p (u_p)|^p \le  \int_{\Omega}  h(u_{p})f \psi_{p}(u_{p})\,.
	\end{equation}
\end{lemma}
\begin{proof}

We divide the proof in few steps. 

\medskip 

{\it Proof of  \eqref{exstima}.}	We consider the solutions $u_{n}$ to \eqref{pbplapapprox}. We apply \eqref{priorigk} to $u_{n}$ and we apply the H\"older inequality to obtain  
	\begin{align*}
	\int_{\Omega}|\nabla G_k(u_{n})|^{p} &\leq  ((h(\infty)+\epsilon_k)||f||_{L^{N}(\Omega)}\mathcal{S}_p)^\frac{p}{p-1} |
	\Omega|. 
	\end{align*}
	Recalling that  $||f||_{L^{N}(\Omega)}< \frac{1}{\mathcal{S}_1 h(\infty)}$, we  fix a constant $c$ such that $h(\infty)||f||_{L^{N}(\Omega)} \mathcal{S}_1<c<1$. By continuity,  there exist $p_0$ sufficiently near to $1^{+}$ and $\overline{k}$ large enough  such that,  $$(h(\infty)+\epsilon_k)||f||_{L^{N}(\Omega)} \mathcal{S}_p<c<1,$$ for any $p\in (1,p_0)$ and $k\geq \overline{k}$. 
	In particular, 
	\begin{align}\label{priorigkN}
	\int_{\Omega}|\nabla G_{\overline k}(u_{n})|^{p} \le c^{\frac{p-1}{p}}|\Omega|.
	\end{align}							
	
	To prove \eqref{exstima} we reason as in the proof of \eqref{prioritk4} on $T_{\overline k}(u_{n})$, and, again by H\"older's inequality, we obtain
	
	\begin{align*}\label{prioritk5}
	\int_{\omega}|\nabla T_{\overline k}(u_{n})|^{p}&\le  4 {\overline k}^p c_{\omega} [(2(p-1))^{p-1}+1]|\Omega|
	\\ \notag
	&+ 4(p-1) [(h(\infty)+
	\epsilon_{\overline k})||f||_{L^{N}(\Omega)} \mathcal{S}_p ]^{\frac{p}{p-1}}|
	\Omega|.
	\end{align*}
	Now observe that, thanks to the choice of $p_{0}$ and $\overline{k}$ all terms at the right hand side of the previous expression are bounded uniformly with respect to $p\in (1,p_0)$. This fact together with  \eqref{priorigkN}, and using weak lower semicontinuity,   shows that  \eqref{exstima} holds.
	
	\medskip 
	
	{\it Global estimate \eqref{exbordo} }.  In order to show \eqref{exbordo} we take $u_{n}^\sigma$ as a test function in \eqref{prioriweakdef} obtaining (for $k_1>k_0$)
\begin{equation}
	\begin{aligned}\label{stimapotenza}
	\displaystyle  \left(\frac{p}{\sigma-1+p}\right)^p\sigma\int_{\Omega}|\nabla u_{n}^{\frac{\sigma-1+p}{p}}|^p &= \int_{\Omega}h_{n}(u_{n})fu_{n}^\sigma \le c_{1}k_0^{\sigma-\gamma}\int_{\{u_{n}\le k_0\}}f
	\\
	& + \left(\max_{s\in [k_0,k_1]}h (s)\right) \ k_1^\sigma \int_{\{k_0<u_{n}< k_1\}} f 
	\\
	&+  (h(\infty)+\epsilon_{k_1})\int_{\{u_{n}\ge k_1\}}fu_{n}^\sigma.
	\end{aligned}
\end{equation}
	We estimate the last term in the right hand side of \eqref{stimapotenza}.  Observing that $\frac{\sigma N}{N-1}<\frac{(\sigma-1+p)p^*}{p}$, one can apply H\"older's inequality and then Young's inequality to get
	\begin{align*}
	(h(\infty)+\epsilon_{k_1})\int_{\{u_{n}\ge k_1\}}fu_{n}^\sigma &\le  (h(\infty)+\epsilon_{k_1})||f||_{L^N(\Omega)} \left(\int_{\{u_{n}\ge k_1\}}u_{n}^{\frac{\sigma N}{N-1}}\right)^{\frac{N-1}{N}} 
	\\
	&\le (h(\infty)+\epsilon_{k_1})||f||_{L^N(\Omega)} \left(\int_{\Omega}u_{n}^{\frac{(\sigma-1+p)p^*}{p}}\right)^{\frac{p\sigma}{(\sigma-1+p)p^*}} |\Omega|^{\frac{(p-1)(\sigma -1+N)}{N(\sigma-1+p)}} 
	\\ 
	&\le (h(\infty)+\epsilon_{k_1})||f||_{L^N(\Omega)} \left(\int_{\Omega}u_{n}^{\frac{(\sigma-1+p)p^*}{p}}\right)^{\frac{p}{p^*}}
	\\&+  (h(\infty)+\epsilon_{k_1})||f||_{L^N(\Omega)} |\Omega|^{\frac{\sigma -1+N}{N}} .	 	
	\end{align*}
	Concerning  the left hand side of \eqref{stimapotenza} we apply the Sobolev inequality and we  have
	\begin{align*}
	\displaystyle  \left(\frac{p}{\sigma-1+p}\right)^p\sigma\int_{\Omega}|\nabla u_{n}^{\frac{\sigma-1+p}{p}}|^p \ge \displaystyle  \left(\frac{p}{\sigma-1+p}\right)^p\frac{\sigma}{ \mathcal{S}_p^{p}}\left(\int_{\Omega}u_{n}^{\frac{(\sigma-1+p)p^*}{p}}\right)^{\frac{p}{p^*}} \,.
	\end{align*}	 	
	Collecting the previous two inequalities gathered with  \eqref{stimapotenza} we deduce 
	\begin{align}\label{stimapotenza2}
	\displaystyle  \left(\left(\frac{p}{\sigma-1+p}\right)^p\frac{\sigma}{ \mathcal{S}_p^{p}} - (h(\infty)+\epsilon_{k_1})||f||_{L^N(\Omega)}  \right) \left(\int_{\Omega}u_{n}^{\frac{(\sigma-1+p)p^*}{p}}\right)^{\frac{p}{p^*}}\le C\,.
	\end{align}
	Now, since $\displaystyle ||f||_{L^N(\Omega)}< \frac{1}{\mathcal{S}_1h(\infty)}$, for $p$ sufficiently near to $1$ and  $k_1$ sufficiently large, one has that   $$\displaystyle \left(\left(\frac{p}{\sigma-1+p}\right)^p\frac{\sigma}{ \mathcal{S}_p^{p}} - (h(\infty)+\epsilon_{k_1})||f||_{L^N(\Omega)}  \right)>c>0\,,$$ 
	for some constant $c$ not depending on both $p<p_{0}$ and $k_{1}$. 
	Therefore, using  \eqref{stimapotenza2}  we deduce
	
		\begin{equation}\label{exbordon}
		||u_n^{\frac{\sigma-1+p}{p}}||_{W^{1,p}_0(\Omega)} \le C(\mathcal{S}_1, \sup_{s\in[k_{0},\infty)}h(s), ||f||_{L^{N}(\Omega)},c_1,  |\Omega|)\,,
	\end{equation} from which \eqref{exbordo} follows  by weak lower semicontinuity.

	\medskip 
				
	{\it Proof of estimate  \eqref{exubounded}}.  	 To show \eqref{exubounded}  we take $G_k(u_{n})$ as a test function in \eqref{pbplapapprox} obtaining
	\begin{equation*}
	\int_{\Omega} |\nabla G_k(u_{n})|^p \le (h(\infty)+\epsilon_k) \int_{\Omega} f G_k(u_{n}).
	\end{equation*}
	Moreover, by the Sobolev, the Young and the H\"older inequalities, we have (recall $p>1$)
	\begin{align*}
	\frac{1}{\mathcal{S}_1} \left(\int_{\Omega}G_k^{\frac{N}{N-1}}(u_{n})\right)^\frac{N-1}{N} &\le \int_{\Omega} |\nabla G_k(u_{n})| \le \frac{1}{p} \int_{\Omega} |\nabla G_k(u_{n})|^p  + \frac{p-1}{p} |\Omega| 
	\\
	&\le (h(\infty)+\epsilon_k) \int_{\Omega} f G_k(u_{n}) + \frac{p-1}{p} |\Omega| 
	\\
	&\le (h(\infty)+\epsilon_k) ||f||_{L^N(\Omega)}\left(\int_{\Omega}G_k^{\frac{N}{N-1}}(u_{n})\right)^\frac{N-1}{N} + \frac{p-1}{p} |\Omega| ,   
	\end{align*}
	that implies
	\begin{equation*}
	\left(\frac{1}{\mathcal{S}_1} - (h(\infty)+\epsilon_k) ||f||_{L^N(\Omega)} \right) ||G_k(u_{n})||_{L^{\frac{N}{N-1}}(\Omega)} \le \frac{p-1}{p} |\Omega|.
	\end{equation*}
	As before, recalling  $||f||_{L^N(\Omega)}< \frac{1}{\mathcal{S}_1 h(\infty)}$,  it  is possible to pick $\overline k$ large  enough so that $\frac{1}{\mathcal{S}_1} - (h(\infty)+\epsilon_{k}) ||f||_{L^N(\Omega)}>c>0$, for any $k\geq \overline k$, where $c$ only depends on $\mathcal{S}_1, h(\infty)$, and  $||f||_{L^{N}(\Omega)}$. Then \eqref{exubounded} follows by Fatou's lemma.

	\medskip 				
				{\it Proof of \eqref{dispotenza}}. Observe that by  \eqref{exbordon} one readily gets, by compact embeddings,  that  $u_{n}^{\frac{\sigma-1+p}{p}}$ strongly converges to $u_{p}^{\frac{\sigma-1+p}{p}}$ in $L^{\frac{N}{N-1}}(\Omega)$. 
				As already done, by  possibly decreasing its value we assume, without loss of generality,  that $k_0\not\in \{\eta: |\{u_p=\eta\}|>0\}$.  
				Recalling \eqref{psi},  we consider  $\psi_{p}(u_n)$ as a test function for \eqref{pbplapapprox} getting 
				$$
				\int_{\Omega} |\nabla \Gamma_p (u_n)|^p \bk= \int_{\Omega}  h_{n}(u_{n})f \psi_{p}(u_{n})\leq c_{1} k_0^{\frac{\sigma+p-1 -\gamma p}{p}} \int_{\{u_n\leq  k_0\}} f +\sup_{s\in [k_{0},\infty)} h(s)\int_{\{u_n> k_0\}} f u_{n}^{\frac{\sigma-1+p}{p}}\leq C\,.
				$$
				
				In particular $\Gamma_{p}(u_{n})$ is bounded in $W^{1,p}_{0}(\Omega)$ with respect to $n$ and we can use weak lower semicontinuity of the norm in order to get 
				$$
\int_{\Omega} |\nabla \Gamma_p (u_p)|^p \le \liminf_{n\to\infty} \int_{\Omega}  h_n(u_{n})f \psi_{p}(u_{n})				
				$$
	What is left is to identify the  limit, as $n$ goes to infinity, of the right hand side of the previous expression.

We write
$$
h_n(u_n)f\psi_{p}(u_n)=h_n(u_n)f\psi_{p}(u_n)\chi_{\{u_n\le k_0\}}+h_n(u_n)fu_{n}^{\frac{\sigma-1+p}{p}}\chi_{\{u_n> k_0\}}\,;
$$
as
$$
h_n(u_n)f\psi_{p}(u_n)\chi_{\{u_n\le k_0\}}\leq c_{1} k_0^{\frac{\sigma+p-1 -\gamma p}{p}}f, 
$$
we can pass to the limit in the first term by dominated convergence. On the other hand, as  	$fu_{n}^{\frac{\sigma-1+p}{p}}$ strongly converges in $L^{1}(\Omega)$ to $fu_{p}^{\frac{\sigma-1+p}{p}}$ and  $h_n(u_n)\chi_{\{u_n> k_0\}}$		converges to $h (u_p)\chi_{\{u_p> k_0\}}$ both a. e. and $\ast$-weak in $L^{\infty}(\Omega)$  we get 
$$
\lim_{n\to\infty} \int_{\Omega}  h_n(u_{n})f \psi_{p}(u_{n})=  \int_{\Omega}  h(u_{p})f \psi_{p}(u_{p})\,,
$$
and so \eqref{dispotenza}. 
						
				\end{proof}	
	
		\begin{remark}
			Let us note that the proof of Theorem \ref{existence_p} keeps working  even for a problems as
			\begin{equation*}\label{pbplapgen}
			\begin{cases}
			\displaystyle - \operatorname{div}(a(x,\nabla u_p))= h(u_p)f &  \text{in}\, \Omega, \\
			u_p=0 & \text{on}\ \partial \Omega,
			\end{cases}
			\end{equation*}
			where $\displaystyle{a(x,\xi):\Omega\times\mathbb{R}^{N} \to \mathbb{R}^{N}}$ is a classical Leray-Lions operator satisfying the following structure conditions
			\begin{align*}
			&a(x,\xi)\cdot\xi\ge \alpha|\xi|^{p}, \ \ \ \alpha>0,
			\\
			&|a(x,\xi)|\le \beta|\xi|^{p-1}, \ \ \ \beta>0,
			\\
			&(a(x,\xi) - a(x,\xi^{'} )) \cdot (\xi -\xi^{'}) > 0,	
			\end{align*}
			for every $\xi\neq\xi^{'}$ in $\mathbb{R}^N$ and for almost every $x$ in $\Omega$.
		\end{remark}

\section{The limit as $p\to 1^+$ for a positive $f$}
\setcounter{equation}{0}

		\noindent In this section we prove our main results concerning the case $p=1$, namely Theorems \ref{existence}, \ref{regularity} and \ref{uniqueness}.  As before, throughout this section,  $h$ satisfies \eqref{h1} and \eqref{h2} and $f\in L^N(\Omega)$ is positive. The proofs of Theorems \ref{existence} and \ref{regularity} will be split  into those of  various lemmata.

	Preliminarily, let us recall that	the solutions $u_p$ found in  Theorem \ref{existence_p} satisfy \eqref{exstima}. This implies that $u_p$ is locally uniformly  bounded in $BV(\Omega)$  with respect to $p$ \bk. Indeed 
		for every $\omega\subset\subset \Omega$
				\begin{equation*}
				\begin{aligned}\label{stimabvloc}
					\int_{\omega} |\nabla u_p| \le \frac{1}{p} \int_{\omega} |\nabla u_p|^p  + \frac{p-1}{p} |\omega|\le C \,.		   
				\end{aligned}	
				\end{equation*}		 
		Also recalling \eqref{exubounded},  by compactness in $BV$, and a  standard diagonal argument,  we deduce the existence of a  function $u\in BV_{\rm{loc}}(\Omega)$ such that (up to not relabeled subsequences)  
				\begin{equation}
				\begin{array}{l}
					u_p \to u \text{ in  $L^q(\Omega)$\bk \ with $q<\frac{N}{N-1}$  and a.e. in $\Omega$},  
					\\ \\ \dys  \text{$\nabla u_p \to D u$ locally $\ast$-weakly as measures\,,} 
				\end{array}\label{up}\end{equation}
				as $p\to 1^+$. First we have the following
				 \begin{lemma}\label{lemmaregolarita}
				 	Let $0 < f\in L^{N}(\Omega)$ such that $\displaystyle ||f||_{L^N(\Omega)}< \frac{1}{\mathcal{S}_1 h(\infty)}$ with $h$ satisfying both \eqref{h1} and \eqref{h2}. Then $u$, defined by \eqref{up}, belongs to $L^\infty(\Omega)$. Moreover $u^\sigma \in BV(\Omega)$. 
				 \end{lemma}
	    \begin{proof}	
		Using the Fatou lemma in \eqref{exubounded} we have that there exists $\overline k$ such that for every $k\ge \overline k$ 
		$$||G_k(u)||_{L^{\frac{N}{N-1}}(\Omega)}=0,$$
		that is  $0\le u \le k$ a.e. in $\Omega$.
	Moreover it follows from the Young inequality and from \eqref{exbordo} that
				\begin{align*}\label{stimabv}
					\displaystyle   \int_{\Omega}|\nabla u_p^{\frac{\sigma-1+p}{p}}| \le \frac{1}{p}\int_{\Omega}|\nabla u_p^{\frac{\sigma-1+p}{p}}|^p + \frac{1}{p^{'}}|\Omega|\le C,
				\end{align*} 
		for some constant $C$ not depending on $p$. This implies that $u_{p}^{\frac{\sigma-1+p}{p}}$ is bounded in $BV(\Omega)$. Then there exists $w\in BV(\Omega)$ such that 
	  	$u_{p}^{\frac{\sigma-1+p}{p}}$ converges to $w$ in $L^q(\Omega)$ for $q<\frac{N}{N-1}$ and a.e. in $\Omega$,  and $\nabla u_{p}^{\frac{\sigma-1+p}{p}}$ converges $D w$ $\ast$-weakly as measures. As $u_p$ converges a.e. to $u$ this implies that $w=u^\sigma$ which concludes the proof.  We stress that we have just shown that 
	  	\begin{equation}\label{convpotenza}
	  		u_{p}^{\frac{\sigma-1+p}{p}} \to u^\sigma \text{ in } L^q(\Omega), \ \ \text{ for every }q<\frac{N}{N-1}.
	  	\end{equation} 
		\end{proof}
	 	 
The following Lemma shows the existence (and the identification) of the vector  field $z$. 
				\begin{lemma}\label{lemma_campoz}
		 	 	 	Under the same assumptions of Lemma \ref{lemmaregolarita},  there exists $z\in \mathcal{D}\mathcal{M}^\infty_{\rm{loc}}(\Omega)$ with $||z||_{\infty}\le 1$ such that 
		 	 	 	\begin{equation}\label{lemma_l1loc}
		 	 	 		h(u)f \in L^1_{\rm loc}(\Omega),
		 	 	 	\end{equation}
			 	 	 	\begin{equation}\label{lemma_distr}
			 	 	 		-\operatorname{div}z = h(u)f \text{  in  } \mathcal{D'}(\Omega).
			 	 	 	\end{equation}
		 	 	 	Moreover
			 	 	 	\begin{equation}\label{lemma_identificarez}
				 	 	 	(z,Du)=|Du| \ \ \ \ \text{as measures in } \Omega.
			 	 	  	\end{equation}		 	 	 		 	 	 	
			 	 \end{lemma}
				 \begin{proof}
				 
We divide the proof into few  steps. 

\medskip 				 
{\it Existence of the field $z$.} Recalling \eqref{exstima} we have, for $1\le q<p'$ and for any $\omega\subset\subset\Omega$
			\begin{align*}
					\int_\omega \left||\nabla u_p|^{p-2}\nabla u_p\right|^q = \int_\omega |\nabla u_p|^{q(p-1)}\le \left(\int_\omega |\nabla u_p|^p\right)^{\frac{q}{p'}}|\omega|^{1-\frac{q}{p'}}\leq C_\omega^{\frac{q}{p'}}|\Omega|^{1-\frac {q}{p'}}		
			\end{align*}
		and thus,
			\begin{equation}\label{stimaz}
				|||\nabla u_p|^{p-2}\nabla u_p||_{q,\omega}\le C_\omega^{\frac{1}{p'}}|\Omega|^{\frac1q-\frac 1{p'}}.
			\end{equation}
		The previous implies that $|\nabla u_p|^{p-2}\nabla u_p$ is bounded in $L^q(\omega,\mathbb{R}^N)$ with respect to $p$. Then there exists $
		z_q\in L^q(\omega,\mathbb{R}^N)$ such that
			$$	|\nabla u_p|^{p-2}\nabla u_p\rightharpoonup z_q\,,\quad\hbox{weakly in }L^q(\omega,\mathbb{R}^N).$$
		A standard  diagonal argument shows that there exists a unique vector field $z$ which is defined on $\Omega$ independently of $q$, such that 
			\begin{equation}\label{convergence2}
				|\nabla u_p|^{p-2}\nabla u_p\rightharpoonup z\,,\quad\hbox{weakly in } L^q(\Omega,\mathbb{R}^N)\,,\quad\forall q<\infty\,.
			\end{equation}
		Moreover, it follows from the lower semicontinuity in \eqref{stimaz} with respect to $p$ that
		$$		||z||_{q,\omega}\le |\Omega|^{\frac1q}\,,\quad\forall q<\infty\,$$
		and thus if $q\to \infty$ then $z\in L^\infty(\omega,\mathbb{R}^N)$ and $||z||_{\infty,\omega}\le 1$. Since this estimate is independent of $\omega$, then $||z||_{\infty} \le 1$.
		
		\medskip 
		
		{\it Distributional formulation.} We prove that \eqref{lemma_distr} holds. Note that by \eqref{convergence2} one can pass to limit with respect to $p$ in the left hand side of the distributional formulation of \eqref{pbplap}.  Concerning the right hand side, we first notice that,  if $h(0)<\infty$,  then one passes to the limit using the a.e. convergence of $u_{p}$ and the Lebesgue theorem. We then assume that $\displaystyle h(0)=\infty$. 
		
		Let $0\le \varphi\in C^1_c(\Omega)$, then applying the Young inequality in the distributional  formulation of \eqref{pbplap} it yields
			\begin{equation*}\label{stima_locale}
				\int_{\Omega} h(u_p)f\varphi \le \frac{1}{p'}\int_{supp \,\varphi} |\nabla u_p|^p + \frac{1}{p}\int_{\Omega} |\nabla\varphi|^{p} \le C \,, 
			\end{equation*}
		 and then the Fatou lemma  gives
			\begin{equation}\label{l1loc}
				\int_{\Omega} h(u)f\varphi \le \liminf_{p\to 1^{+}} \int_{\Omega} h(u_p)f\varphi \le C\,,  
			\end{equation}		
		which implies \eqref{lemma_l1loc} and that $u>0$ a.e. in $\Omega$ since $f>0$ a.e. in $\Omega$. 
		 Now, recalling Remark \ref{remtest}, we are able to take $V_\delta(u_p)\varphi$ as a test function in \eqref{pbplap} where $0\le \varphi\in C^1_c(\Omega)$  and $V_{\delta}$ is defined as in \eqref{V_delta},\bk we have 
				\begin{equation*}
					\int_{\Omega} |\nabla u_p|^p V^{'}_\delta(u_p)\varphi + \int_{\Omega} |\nabla u_p|^{p-2}\nabla u_p \cdot \nabla \varphi V_\delta(u_p)= \int_{\Omega} h(u_p)f V_\delta(u_p)\varphi,
			 	\end{equation*}
		which, since $V'_\delta(s)\le 0$, takes to
	 	 	 	\begin{equation*}\label{vicinozero1}
		 	 	 	\int_{\{u_p\le \delta\}}  h(u_p)f\varphi \le\int_{\Omega} |\nabla u_p|^{p-2}\nabla u_p \cdot \nabla \varphi V_\delta(u_p).
	 	 	 	\end{equation*}
 	 	 Hence  (recall \eqref{convergence2})  we have 
	 	 	 	\begin{equation*}\label{vicinozero2}
		 	 	 	\limsup_{p\to 1^+}\int_{\{u_p\le \delta\}}  h(u_p)f\varphi \le\int_{\Omega} z \cdot \nabla \varphi V_\delta(u).
	 	 	 	 \end{equation*}
		Then the Lebesgue  theorem implies that 
		 	 	\begin{equation}\label{vicinozero3}
			 	 	\lim_{\delta\to 0^+}\limsup_{p\to 1^+}\int_{\{u_p\le \delta\}}  h(u_p)f\varphi \le\int_{\{u=0\}} z \cdot \nabla \varphi=0,
		 	    \end{equation}		 	 	 	 
		since $u>0$ a.e. in $\Omega$.  Observe now that, by a standard density argument, \eqref{vicinozero3} can be shown to hold for any $\varphi\in C^1_c(\Omega)$\bk.
		
		In order  to pass to the limit (with respect to $p$) in  
		$$
		 \int_{\Omega} |\nabla u_p|^{p-2}\nabla u_p \cdot \nabla \varphi = \int_{\Omega}  h(u_p)f \varphi\,,
		 $$
		 we then let 
		 	 	\begin{equation*}\label{passaggioalimite}
			 	 	 \int_{\Omega}  h(u_p)f \varphi= \int_{\{u_p\le \delta\}}  h(u_p)f \varphi + \int_{\{u_p>\delta\}}  h(u_p)f \varphi\,, 
			 	\end{equation*}	
				and, using  the same agreement \bk on $\delta$ used in the proof of Theorem \ref{existence_p} (namely $\delta \notin \{\eta: |\{u=\eta\}|>0\}$),  and recalling \eqref{l1loc} we have, again by Lebesgue theorem that 	 	 	    
		$$
\lim_{\delta\to 0^+}\lim_{p\to1^{+}}\int_{\{u_p>\delta\}}  h(u_p)f \varphi=\int_{\Omega}  h(u)f \varphi\,,
		$$
		
		for any $\varphi \in C^{1}_{c}(\Omega)$. This fact, together with \eqref{vicinozero3} implies \eqref{lemma_distr}. Observe that this also implies that  $z\in\mathcal{D}\mathcal{M}^\infty_{\rm{loc}}(\Omega)$.
		
		\medskip

	 	{\it A variational identity.} In order to show \eqref{lemma_identificarez}, the first step consists in proving  the following
	 		 \begin{equation}\label{equ}
	 		 -(u^\sigma)^*\operatorname{div}z =  h(u)f u^\sigma \text{  in  } \mathcal{D'}(\Omega).
	 		 \end{equation}	
		To do that we re-adapt the  idea in \cite{DGS};  we  test \eqref{lemma_distr} with $(\rho_\epsilon\ast u^\sigma)\varphi$, where $\rho_\epsilon$ is a standard mollifier and $\varphi\in C^1_c(\Omega)$.
		One has 
			 \begin{equation}\label{eqsigma}
				 -\int_\Omega (\rho_\epsilon\ast u^\sigma)\varphi\, \operatorname{div}z =  \int_\Omega h(u)f(\rho_\epsilon\ast u^\sigma)\varphi.
			 \end{equation}
		Since $u\in L^\infty(\Omega)$ then we have $(u^\sigma)^*\le ||u^\sigma||_{L^\infty(\Omega)}$ $\mathcal H^{N-1}$-a.e. and so $\operatorname{div}z$-a.e. (recall that $\operatorname{div}z <\!\!<\mathcal H^{N-1}$). It is then standard (see for instance Propositions 3.64 (b) and  3.69 (b) of \cite{AFP}) that $\rho_\epsilon\ast u^\sigma \to (u^\sigma)^*   \quad\mathcal H^{N-1}$- a.e. and then $\operatorname{div}z$- a.e.
	  Therefore, we can pass to the limit in both sides of \eqref{eqsigma} by  dominated convergence theorem also using \eqref{l1loc} and the fact that
		$|\rho_\epsilon\ast u^\sigma|\leq ||u^\sigma||_{L^\infty(\Omega)}$. Then  \eqref{equ} holds.
		
		\medskip 
		
	{\it A first identification result.} A second ingredient for \eqref{lemma_identificarez} is the following identification identity  involving $u^{\sigma}$:
	\begin{equation*}\label{ii}
	|D u^{\sigma}|  =  (z, D u^\sigma) \ \ \text{as measures.}
	\end{equation*}
	  We use $u_p^\frac{\sigma-1+p}{p}\varphi$ as a test function in \eqref{pbplap} where $0\le \varphi \in C^1_c(\Omega)$,  and we get
	\begin{equation*}	
	\left(\frac{\sigma-1+p}{p}\right)\left(\frac{p^2}{\sigma-1+p^2}\right)^p\int_{\Omega} \varphi|\nabla u_p^{\frac{\sigma-1+p^2}{p^2}}|^{p} + \int_{\Omega} u_p^\frac{\sigma-1+p}{p}|\nabla u_p|^{p-2}\nabla u_p \cdot \nabla \varphi = \int_{\Omega}  h(u_p)f u_p^\frac{\sigma-1+p}{p} \varphi.	
	\end{equation*}
	Thus, by Young's inequality, we deduce
	\begin{equation}
	\begin{aligned}\label{z_1}
	\	\left(\frac{\sigma-1+p}{p}\right)^{\frac{1}{p}}\left(\frac{p^2}{\sigma-1+p^2}\right)\int_{\Omega} \varphi|\nabla u_p^{\frac{\sigma-1+p^2}{p^2}}| &+\int_{\Omega} u_p^\frac{\sigma-1+p}{p}|\nabla u_p|^{p-2}\nabla u_p \cdot \nabla \varphi 
	\\ 
	&\le \int_{\Omega}  h(u_p)f u_p^\frac{\sigma-1+p}{p} \varphi + \frac{p-1}{p}\int_{\Omega}\varphi.	
	\end{aligned}
	\end{equation} 
	By \eqref{convpotenza} we observe that  $u_p^\frac{\sigma-1+p}{p}$ converges to $u^\sigma$ in $L^r(\Omega)$ for $r<\frac{N}{N-1}$ and a.e. in $\Omega$ and (see \eqref{convergence2}) that  $|\nabla u_p|^{p-2}\nabla u_p$ converges weakly to $z$ in $L^q(\Omega,\mathbb{R}^N)$ for any  $q<\infty$;  this is sufficient to pass to the limit in the second term on the left hand side of the previous.

	Concerning  the right hand side of \eqref{z_1} we have, for  $\delta>0$,
	\begin{equation*}\label{stimaloc1}
		\int_{\Omega}  h(u_p)f u_p^\frac{\sigma-1+p}{p} \varphi= \int_{\{u_p\le \delta\}}  h(u_p)f u_p^{\frac{\sigma-1+p}{p}} \varphi+\int_{\{u_p>\delta\}}  h(u_p) fu_p^\frac{\sigma-1+p}{p} \varphi.
	\end{equation*}	
	On one hand, we treat the first term on the right hand side of the previous as follows
	\begin{equation*}\label{stimaloc12}
	\lim_{\delta\to 0^+}\limsup_{p\to 1^{+}}\int_{\{u_p\le \delta\}}  h(u_p)f u_p^{\frac{\sigma-1+p}{p}}\varphi \le \lim_{\delta\to 0^+}\limsup_{p\to 1^{+}}\delta^{\frac{\sigma-1+p}{p}}\int_{\{u_p\le \delta\}}  h(u_p)f\varphi = 0\,.
	\end{equation*}	
reasoning as for \eqref{vicinozero3}.  
	For the second term, using the usual convention for the choice of $\delta$,  the a.e. convergence and the weak convergence of $h(u_p) u_p^{\frac{\sigma-1+p}{p}}\chi_{\{u_p>\delta\}}$ to $h(u) u^{\sigma}\chi_{\{u>\delta\}}$ in $L^{\frac{N}{N-1}}(\Omega)$, with respect to $p$, recalling that $f\in L^{N}(\Omega)$, implies that 
	\begin{equation*}
	\displaystyle \lim_{p\to 1^{+}}\int_{\{u_p>\delta\}}  h(u_p)f u_p^\frac{\sigma-1+p}{p} \varphi = \int_{\{u>\delta\}}  h(u)f u^\sigma \varphi.
	\end{equation*}	
	Then by the Lebesgue theorem, we can pass to the limit also as $\delta\to 0^+$  since $h(u)fu^\sigma \varphi  \in L^1(\Omega)$ (recall $u$ is bounded). In particular observe that, for fixed $0\le \varphi \in C^1_c(\Omega)$,  all but the first term  in \eqref{z_1} are uniformly bounded with respect to $p$. 

	Therefore $u_p^{\frac{\sigma-1+p^2}{p^2}}$ is bounded in $BV_{\rm{loc}}(\Omega)$ and it locally converges, up to subsequences, to $u^{\sigma}$ a.e. in $\Omega$, in $L^r(\Omega)$ with $r<\frac{N}{N-1}$ and $\nabla u_p^{\frac{\sigma-1+p^2}{p^2}}$ converges $\ast$-weakly locally as measures to $D u^\sigma$. Then by weak lower semicontinuity in the first term we obtain, recalling \eqref{equ}, that 
				\begin{equation*}
		 	 	 	\int_{\Omega} \varphi|D u^{\sigma}| + \int_{\Omega} u^\sigma z \cdot \nabla \varphi \le \int_{\Omega}  h(u) f u^\sigma \varphi = -\int_{\Omega}(u^{\sigma})^*\varphi \operatorname{div}z, \ \ \ \forall \varphi\in C^1_c(\Omega), \ \ \varphi \ge 0.	
	 	 	 	\end{equation*}	
	 	 	 		 	 	 	
 	 	Moreover it follows from Proposition \ref{nuova} that 
	 		 	\begin{equation*}
	 		 	 	\int_{\Omega} \varphi|D u^{\sigma}| \le - \int_{\Omega} u^\sigma z \cdot \nabla \varphi -\int_{\Omega}(u^{\sigma})^* \varphi\operatorname{div}z = \int_{\Omega}\varphi (z, D u^\sigma),  \ \ \ \forall \varphi\in C^1_c(\Omega), \ \ \varphi \ge 0\,,	
	 	  	 	\end{equation*}			 	 	 then 
\begin{equation}	\label{rev}			\int_{\Omega} \varphi|D u^{\sigma}|  = \int_{\Omega}\varphi (z, D u^\sigma),  \ \ \ \forall \varphi\in C^1_c(\Omega), \ \ \varphi \ge 0
\,,
\end{equation}	 
		the reverse inequality being  trivial since $||z||_{\infty}\le 1$.  
		\medskip 

	{\it Proof completed.}	 Here we show that \eqref{rev}  implies \eqref{lemma_identificarez}. 
Let us choose in \eqref{radonnik} $\Lambda: [0,\infty)\to[0,\infty)$ defined by $\Lambda(s)=s^\sigma$, which is a Lipschitz increasing function, since $\sigma\geq 1$. Then
	\begin{equation*}
	\frac{(z, Du)}{|Du|} = \theta(z,Du,x)=\theta(z,Du^\sigma,x) =\frac{(z, Du^\sigma)}{|Du^\sigma|}  \qquad \text{for \(|Du|\)-a.e.}\ x\in\Omega, 
	\end{equation*}	
	namely \eqref{lemma_identificarez}.	

		 	 	 \end{proof}		 	 	  
	In order to prove Theorem \ref{existence} we need to show that $z\in \mathcal{D}\mathcal{M}^\infty(\Omega)$. In fact, we have the following stronger general fact:			
	\begin{lemma}\label{lemmal1}
		Let $0\le g\in L^{1}_{\rm{loc}}(\Omega)$ and let $z\in \mathcal{D}\mathcal{M}^\infty_{\rm{loc}}(\Omega)$ with $||z||_{\infty}\le 1$ such that 
		\begin{equation}\label{lemma_distr*}
		-\operatorname{div}z = g \text{  in  } \mathcal{D'}(\Omega),
		\end{equation}
		then 
		\begin{equation}\label{l1}
		g \in L^1(\Omega)\,.
		\end{equation}	
		In particular,  $\operatorname{div}z \in L^1(\Omega)$ and the following holds  	
		\begin{equation}\label{eqtestbv}
		-\int_\Omega v\,\Div z = \int_\Omega  gv, \ \ \forall v\in BV(\Omega)\cap L^\infty(\Omega).
		\end{equation}
	\end{lemma}
	\begin{proof}
		Let $0\le v\in W^{1,1}_0(\Omega)$ and let $\varphi_n\in C^{1}_c(\Omega)$ be a sequence of nonnegative functions converging in $W^{1,1}_0(\Omega)$ to $v$. Let us take $\rho_\eta*(v\wedge \varphi_n)$ as test function in \eqref{lemma_distr*} where $\rho_\eta$ ($\eta>0$) is a standard mollifier. We obtain
		\begin{equation}\label{contest}
		\int_\Omega z \cdot \nabla (\rho_\eta*(v\wedge \varphi_n)) = \int_\Omega g\rho_\eta*(v\wedge \varphi_n),
		\end{equation}
		and we are able to pass to the limit in the left hand side of the previous as $\eta\to 0$ since $\rho_\eta*(v\wedge \varphi_n)\to v\wedge \varphi_n$ strongly in $W^{1,1}_0(\Omega)$. For the right hand side we observe that, for $\eta>0$ small enough, $\sop(\rho_\eta*(v\wedge \varphi_n))\subseteq \omega_n$, where $\omega_n\subset\subset\Omega$. Moreover $\|\rho_\eta*(v\wedge \varphi_n)\|_{L^\infty(\Omega)}\leq	\|v\wedge \varphi_n\|_{L^\infty(\Omega)}$ and, as $\eta\to 0$, $\rho_\eta*(v\wedge \varphi_n)$ converges a.e. in $\Omega$ to $v\wedge \varphi_n$. Thus	$\rho_\eta*(v\wedge \varphi_n)$ converges $\ast$-weak in $L^\infty(\Omega)$ to $v\wedge \varphi_n$.
		\\Then, since $g\in L^1(\omega_n)$, we have that
		\begin{equation}\label{contest2}
		\displaystyle \lim_{\eta\to 0}\int_\Omega g (\rho_\eta*(v\wedge \varphi_n))=  \int_\Omega g(v\wedge \varphi_n).
		\end{equation}
		By \eqref{contest} and \eqref{contest2}  we deduce
		\begin{equation}\label{contest3}
		\int_\Omega z \cdot \nabla (v\wedge \varphi_n)
		=
		\int_\Omega g(v\wedge \varphi_n).
		\end{equation}
		Now, we need to pass to the limit the previous as $n\to \infty$. Since $v\wedge \varphi_n$ converges to $v$ in $W^{1,1}_0(\Omega)$ then 
		we have
		\begin{equation*}
		\dys \lim_{n\to\infty}\int_\Omega z \cdot \nabla (v\wedge \varphi_n)
		=
		\int_\Omega z \cdot \nabla v
		.
		\end{equation*}
		For the right hand side of \eqref{contest3} we observe that $g(v\wedge \varphi_n)$ converges a.e. in $\Omega$ to $gv$ and that $		0\leq g(v\wedge \varphi_n)\leq gv.$
		Then, in order to apply the Lebesgue theorem, it is sufficient to show that $gv \in L^1(\Omega)$.
		Indeed we have
		\begin{equation*}
		\int_\Omega g\varphi_n = \int_\Omega z\cdot \nabla \varphi_n \leq \|z\|_{L^\infty(\Omega)}\int_\Omega  |\nabla \varphi_n| \le C\,,
		\end{equation*}
		where the last inequality follows from the fact that $\varphi_n$ converges to $v$ in $W^{1,1}_0(\Omega)$. Then an application of the Fatou lemma implies $gv\in L^1(\Omega)$. Hence we have proved that
		\begin{equation}\label{weak11}
		\int_\Omega z \cdot \nabla v
		=
		\int_\Omega gv, \ \ \forall v\in W^{1,1}_0(\Omega), \ \ v\ge 0.
		\end{equation}				 
		 Now we take $\tilde{v} \in W^{1,1}(\Omega)$ and then it follows from \cite[Lemma 5.5]{An} the existence of $w_n\in
		W^{1, 1}(\Omega)\cap C(\Omega)$ having $w_n|_{\partial\Omega}=\tilde{v}|_{\partial\Omega}$, $\displaystyle\int_\Omega|\nabla
		w_n|\,dx\le\displaystyle\int_{\partial\Omega}\tilde{v}\,d\mathcal H^{N-1}+\frac1n\,,$ and such that $w_n$ tends to $0$ in $\Omega$. 
		%
		%
		Clearly, we can take $|v-w_n|\in W_0^{1,1}(\Omega)$ as a test function in \eqref{weak11}, obtaining
		\begin{align*}
		\int_\Omega g|\tilde{v}-w_n| &= \int_\Omega z\cdot\nabla|\tilde{v}-w_n| \le \|z\|_{\infty}\int_\Omega|\nabla \tilde{v}| + \|z\|_{\infty} \int_\Omega|\nabla w_n| 
		\\
		&\le \int_\Omega|\nabla \tilde{v}|+\int_{\partial\Omega}\tilde{v}\,d\mathcal H^{N-1}+\frac1n\,.
		\end{align*}
		Once again an application of the Fatou lemma implies
		\begin{equation*}
		\int_\Omega g\tilde{v}\le \int_\Omega|\nabla \tilde{v}|+\int_{\partial\Omega}\tilde{v}\,d\mathcal H^{N-1},
		\end{equation*}
		where, taking $\tilde{v}\equiv 1$, one deduces that $g$ belongs to $L^1(\Omega)$.
		Since $\operatorname{div}z \in L^1(\Omega)$, one can apply Anzellotti's theory in order to prove \eqref{eqtestbv}. \bk
	\end{proof}

The following lemma  shows that the boundary datum is attained in the sense of Definition \ref{weakdef}. 
		 	 	  
		 	 	\begin{lemma}\label{lemmabordo}
		 	 		Under the same assumptions of Lemma \ref{lemmaregolarita}, the vector field $z$ found in Lemma \ref{lemma_campoz} is such that  one of the following holds:		 	 	 
		 	 		\begin{align}\label{bordo}
		 	 			& \lim_{\epsilon\to 0} \fint_{\Omega\cap B(x,\epsilon)} u (y) dy = 0 \ \ \ \text{or} \ \ \ [z,\nu] (x)= -1 \ \ \ \text{for  $\mathcal{H}^{N-1}$-a.e. } x \in \partial\Omega.
		 	 		\end{align}
		 	 	\end{lemma}		 	 	  
		 	 	\begin{proof}
		 	 		In order to prove \eqref{bordo} we observe  that \eqref{dispotenza} together with the Young inequality and the fact that 	$u_p^\frac{\sigma-1+p}{p}$	has zero trace in $W^{1,p}_{0}(\Omega)$, gives	 	 	
		 	 		\begin{equation*}
		 	 			\int_{\Omega} |\nabla \Gamma_{p}(u_{p})| + \int_{\partial \Omega}u_p^\frac{\sigma-1+p}{p} d\mathcal{H}^{N-1} \le  \int_{\Omega}  h(u_{p})f \psi_{p}(u_{p}) + \frac{p-1}{p}|\Omega|\,.	
		 	 		\end{equation*}		 	 	 	
		 	 		Now, we use weak lower semicontinuity on the left hand side, while, reasoning as in the proof of Lemma \ref{lemma_campoz} (splitting the integral the function in the two zones $\{u_{p}\leq k_{0}\}$ and $\{u_{p}> k_{0}\}$ and using the definition of $\psi_{p}$) it is not difficult to use  Lebesgue theorem  in order to get 
		 	 		\begin{equation*}
		 	 			\int_{\Omega} |D u^\sigma| + \int_{\partial \Omega}u^\sigma d\mathcal{H}^{N-1} \le  \int_{\Omega}  h(u)f u^\sigma = -\int_{\Omega}(u^{\sigma})^*\operatorname{div}z\,,
		 	 		\end{equation*}	
				where in the last equality we used \eqref{equ}. 		 	 	 	
		 	 		By the Gauss--Green  formula \eqref{GreenIII} we have
		 	 		\begin{equation*}
		 	 			\int_{\Omega} |D u^\sigma| + \int_{\partial \Omega}u^\sigma d\mathcal{H}^{N-1} \le  \int_{\Omega}(z,D u^\sigma) - \int_{\partial \Omega} [u^\sigma z,\nu]d\mathcal{H}^{N-1}.
		 	 		\end{equation*}	
		 	 		 Now we can apply Lemma \ref{lemmal1} with $g=h(u)f$ in order to deduce that $z\in \DM(\Omega)$ and then, since by \eqref{rev} $|D u^\sigma|= (z,D u^\sigma)$, the previous implies \bk
		 	 		\begin{equation*}\label{eqbordo}u^\sigma( 1 + [z,\nu]) = 0  \ \ \ \text{   $\mathcal{H}^{N-1}$- a.e. on $\partial \Omega$}.\end{equation*}
		 	 		Therefore, either   $ [z,\nu](x)=-1\bk$ or $u^{\sigma}(x)=0$ for $\mathcal{H}^{N-1}$-a.e. $x\in \partial\Omega$. In particular (see Theorem 3.87, \cite{AFP}), if $u^{\sigma}(x)=0$,  then 				 $$\lim_{\epsilon\to 0} \fint_{\Omega\cap B(x,\epsilon)} u^{\sigma} (y) dy=0\,.$$
		 	 		If $\sigma>1$,  using H\"older inequality one gets, 
		 	 		$$
		 	 		\begin{array}{l}
		 	 		\dys\fint_{\Omega\cap B(x,\epsilon)} u(y) dy \leq \left( \fint_{\Omega\cap B(x,\epsilon)} u^{\sigma} (y) dy\right)^{\frac{1}{\sigma}}\frac{|\Omega\cap B(x,\epsilon)|^{\frac{1}{\sigma'}}}{\epsilon^{\frac{N}{\sigma'}}}\leq C\left( \fint_{\Omega\cap B(x,\epsilon)} u^{\sigma} (y) dy\right)^{\frac{1}{\sigma}}\stackrel{\epsilon\to 0}{\longrightarrow} 0\,,
		 	 		\end{array} 
		 	 		$$
		 	 		that is \eqref{bordo} holds. 
		 	 	\end{proof}

				\begin{proof}[Proof of Theorem \ref{existence}]
					The proof follows by gathering together  Lemmata \ref{lemmaregolarita} -- \ref{lemmal1}. 					
				\end{proof}

				\begin{proof}[Proof of Theorem \ref{regularity}]
					By Lemma \ref{lemmaregolarita} we  have  that $u^\sigma\in BV(\Omega)$. 
					Moreover, if $h(0)=\infty$, it was already observed that, as \eqref{l1loc} is in force,  then $u>0$ a.e. in $\Omega$.  Moreover,  one can apply Lemma \ref{lemmal1} with $g=h(u)f$ in order to deduce that $\operatorname{div}z \in L^1(\Omega)$ and that \eqref{formtestestese} holds.\bk

				Now let $h(0)<\infty$, we want to show that  $u\equiv 0$. We consider the solution $0\le w_p\in W^{1,p}_0(\Omega)\cap L^\infty(\Omega)$ to 
						\begin{equation}\label{confronto}
							\begin{cases}
								\displaystyle - \Delta_p w_p= ||h||_{L^\infty([0,\infty))}f &  \text{in}\, \Omega, \\
								w_p=0 & \text{on}\ \partial \Omega.
							\end{cases}
						\end{equation}					 
					If $\dys||h||_{L^\infty([0,\infty))}||f||_{L^N(\Omega)}< \frac{1}{\mathcal{S}_1}$ then using \cite[Theorem $4.1$]{CT} we deduce that  $w_p$ goes to zero a.e. in $\Omega$ as $p\to 1^{+}$. 
					\\On the other hand, we recall that $u$ is the a.e. limit in $\Omega$ of the solutions to 
						\begin{equation}\label{confronto2}
							\begin{cases}
								\displaystyle - \Delta_p u_p= h(u_p)f &  \text{in}\, \Omega, \\
								u_p=0 & \text{on}\ \partial \Omega,
							\end{cases}
						\end{equation}
					where $0\le u_p\in W^{1,p}_0(\Omega)\cap L^\infty(\Omega)$. 
					  We take $(w_p-u_p)^-$ as a test function in the difference between weak formulations \eqref{confronto} and \eqref{confronto2}
						\begin{equation*}
							-\int_{\{w_p<u_p\}} (|\nabla w_p|^{p-2}\nabla w_p - |\nabla u_p|^{p-2}\nabla u_p)\cdot \nabla (w_p-u_p) = \int_{\Omega} (||h||_{L^\infty([0,\infty))}- h(u_p))f(w_p-u_p)^- \ge 0,
						\end{equation*}
						which, by monotonicity, implies
						$$\int_{\{w_p<u_p\}} (|\nabla w_p|^{p-2}\nabla w_p - |\nabla u_p|^{p-2}\nabla u_p)\cdot \nabla (w_p-u_p)=0,$$
						that is $w_p\ge u_p\ge 0$ a.e. in $\Omega$. Therefore, taking $p\to 1^+$, one obtains $u\equiv 0$.
						\bk
					\end{proof}

We conclude this section by proving our uniqueness result. 
					
\begin{proof}[Proof of Theorem \ref{uniqueness}]
Since $h(u)f\in L^1_{\rm loc}(\Omega)$ (see \eqref{def_l1loc}) we can apply Lemma \ref{lemmal1} deducing
\begin{equation*}
-\int_\Omega v\,\Div z = \int_\Omega  h(u)fv, \ \ \forall v\in BV(\Omega)\cap L^\infty(\Omega).
\end{equation*}	
Moreover we apply Proposition \ref{poiu} and, recalling \eqref{des2}, we deduce						
\begin{equation}\label{test}
\int_\Omega (z, Dv)-\int_{\partial\Omega}v[z,\nu]\, d\mathcal H^{N-1} = \int_\Omega  h(u)fv,\ \ \forall v\in BV(\Omega)\cap L^\infty(\Omega).
\end{equation}
Let $u_1$ and $u_2$ be solutions to problem \eqref{pb1lap} satisfying \eqref{eqregul} and we denote by, respectively, $z_1$ and $z_2$ the vector fields appearing in Definition \ref{weakdef}. Now we take $v=u_1^\sigma- u_2^\sigma$ in the difference of weak formulations \eqref{test} solved by $u_1,u_2$. Thus 
\begin{equation*}
\begin{aligned}
\int_\Omega (z_1, Du_1^\sigma)-\int_\Omega(z_2, Du_1^\sigma) &+ \int_\Omega(z_2, Du_2^\sigma) - \int_\Omega(z_1, Du_2^\sigma) - \int_{\partial\Omega}(u_1^\sigma-u_2^\sigma)[z_1,\nu])\, d\mathcal H^{N-1}				    	
\\
&+\int_{\partial\Omega}(u_1^\sigma-u_2^\sigma)[z_2,\nu])\, d\mathcal H^{N-1}						    	
= \int_\Omega (h(u_1)- h(u_2))f(u_1^\sigma- u_2^\sigma).
\end{aligned}
\end{equation*} 
Then we can reason as in the last step of the proof Lemma \ref{lemma_campoz} in order to deduce $$\int_\Omega (z_i, Du_i^\sigma) = \int_\Omega |Du_i^\sigma|\ \ \ \text{for}\ \  i=1,2\,.$$ Moreover observe that for a nonnegative function $u\in L^\infty(\Omega)$ such that  $u^{\sigma}$ in $BV(\Omega)$ one has that  $$\dys\lim_{\epsilon\to 0}  \fint_{\Omega\cap B(x,\epsilon)} u (y) dy = 0$$ implies $$\dys\lim_{\epsilon\to 0}  \fint_{\Omega\cap B(x,\epsilon)} u^\sigma (y) dy = 0\,.$$
In particular
$$
u_{i}^{\sigma}(1+[z_{i},\nu])=0\ \ \mathcal H^{N-1}-\text{a.e. on}\ \partial\Omega\  \ \text{for}\ \  i=1,2
$$
Then, it follows
\begin{align*}
\int_\Omega |Du_1^\sigma|-\int_\Omega(z_2, Du_1^\sigma) &+ \int_\Omega|Du_2^\sigma| - \int_\Omega(z_1, Du_2^\sigma) +  \int_{\partial\Omega}(u_1^\sigma +u_1^\sigma [z_2,\nu]  )\, d\mathcal H^{N-1}				    	
\\
	& +\int_{\partial\Omega}(u_2^\sigma[z_1,\nu] + u_2^\sigma)\, d\mathcal H^{N-1}	 					    	
= \int_\Omega (h(u_1)- h(u_2))f(u_1^\sigma- u_2^\sigma).
\end{align*}					    	
Hence recalling that $||z_i||_{\infty} \le 1$ and that $[z_i,\nu] \in [-1,1]$ for $i=1,2$ then the left hand side of the previous is nonnegative. This gives that 
$$\int_\Omega (h(u_1)- h(u_2))f(u_1^\sigma- u_2^\sigma)\ge 0,$$
which implies $u_1=u_2$ a.e. in $\Omega$ since $f>0$ a.e. in $\Omega$.
\end{proof}
					
\section{Nonnegative data $f$}

\label{sec:fnonnegativa}
\setcounter{equation}{0}	

Here we extend existence Theorem \ref{existence}  to the case of a nonnegative $f$ in $L^N(\Omega)$ in   \eqref{pb1lap}.  Here we focus on the purely singular case $h(0)=\infty$; if this is not the case (i.e. $h(0)<\infty$), one can easily re-adapt (with many simplifications) the argument of the previous section in order to  obtain a solution to problem \eqref{pb1lap} which satisfies \eqref{def_l1loc}-\eqref{def_bordo}. 

 As suggested in  \cite{DGS},  when $h$ actually blows up at the origin then the notion of solution should be suitably modified. \bk In fact, roughly speaking, the approximating solutions $u_{p}$ could converge to a limit function $u$  that may have a non-trivial set $\{u=0\}$. This fact, in the $BV$ context amounts to the fact that an additional term (namely a measure) appears in the limit equation (see Remark \ref{remarkequazione}); this additional term can be absorbed in the principal part of the equation by formally multiplying it by $\chi^*_{\{u>0\}}$. Moreover, in this case the vector field $z$ will actually belong to  $\mathcal{D}\mathcal{M}^\infty_{\rm{loc}}(\Omega)$ and this leads to a different formulation for the boundary datum that involves the power $\sigma=\max(1,\gamma)$ of $u$. As a matter of fact, in the case $f>0$, the two definitions do essentially coincide (see Remark \ref{rema}).  We set the following 
		\begin{defin}\label{weakdefnonnegative}
			Let $0\le f\in L^{N}(\Omega)$ then a function $u\in BV_{\rm{loc}}(\Omega)\cap L^\infty(\Omega)$ having $\chi_{\{u>0\}} \in BV_{\rm{loc}}(\Omega)$ and $u^\sigma \in BV(\Omega)$  is a solution to problem \eqref{pb1lap} if there exists $z\in \mathcal{D}\mathcal{M}^\infty_{\rm{loc}}(\Omega)$ with $||z||_{\infty}\le 1$ such that
			\begin{align}
				&h(u)f\in L^{1}_{\rm{loc}}(\Omega),\label{nondef_l1}
				\\
				&-(\operatorname{div}z)\chi^*_{\{u>0\}} = h(u)f \ \ \ \text{in  } \mathcal{D'}(\Omega), \label{nondef_eqdistr}
				\\
				&(z,Du)=|Du| \label{nondef_campo} \ \ \ \ \text{as measures in } \Omega,
				\\
				&  u^\sigma(x) + [u^\sigma z,\nu] (x)= 0 \label{nondef_bordo}\ \ \ \text{for  $\mathcal{H}^{N-1}$-a.e. } x \in \partial\Omega.
			\end{align}
			\end{defin}	
		\begin{remark}\label{remarkequazione}
			 It is worth noting that, reasoning as in \cite{DGS}, one can prove that $(z, D\chi_{\{u>0\}})= |D\chi_{\{u>0\}}|$. This means that, by  the Anzellotti theory, one has
			\begin{equation*}
			-\big(\Div z\big)\chi^*_{\{u>0\}} =-\Div\big(z\chi_{\{u>0\}}\big)+(z,D \chi_{\{u>0\}}),
			\end{equation*} 
			and then the equation \eqref{nondef_eqdistr} reads as 
			\begin{equation*}
			-\Div\big(z\chi_{\{u>0\}}\big)+|D \chi_{\{u>0\}}|=h(u)f,
			\end{equation*}
			that is, it reduces to  a sum of an operator in divergence form and an additional term $|D \chi_{\{u>0\}}|$, which is a measure concentrated on the, non-trivial in this case,  reduced boundary $\partial^*\{u>0\}$, or equivalently on the reduced boundary $\partial^*\{u=0\}$. \bk
		\end{remark}
		\begin{remark}\label{rema} 
			Let us stress that  a solution in the sense of Definition \ref{weakdefnonnegative} is also a solution in the sense of Definition \ref{weakdef} in case of $f>0$ a.e. in $\Omega$. Indeed, since $h(u)f$ is locally integrable then (recall we are assuming $h(0)=\infty$) $u>0$ a.e. in $\Omega$ and \eqref{nondef_eqdistr} reads as
				$$-\operatorname{div}z = h(u)f \ \ \ \text{in  } \mathcal{D'}(\Omega).$$
			Then we can apply Lemma \ref{lemmal1} in order to deduce that $z\in \mathcal{D}\mathcal{M}^\infty(\Omega)$. Finally we only need to show that \eqref{def_bordo} holds. We observe that having $u^\sigma\in BV(\Omega) \cap L^\infty(\Omega)$ and $z\in \mathcal{D}\mathcal{M}^\infty(\Omega)$ we can use \eqref{des2} in order to deduce  from  \eqref{nondef_bordo} that 
			$$	u^\sigma(x)(1 + [z,\nu] (x))= 0 \ \ \ \text{for  $\mathcal{H}^{N-1}$-a.e. } x \in \partial\Omega\,,$$
			that implies \eqref{def_bordo} reasoning as in the proof of Lemma \ref{lemmabordo}.   
			
			On the other hand it is easy to see that a solution in the sense of Definition \ref{weakdef} also satisfies Definition \ref{weakdefnonnegative} provided $u^{\sigma}\in BV(\Omega)$. 
			
			Let us also finally remark that uniqueness of solution in the sense of Definition \ref{weakdefnonnegative} is not expected in general as some one dimensional  examples in the model case with  $\gamma\leq 1$ show (see \cite{DGS}).  			 				
		\end{remark}

	We have the following counterpart of Theorem \ref{existence} for general nonnegative $f$. 
		\begin{theorem}\label{existencenonnegative}
			Let $0\le f\in L^{N}(\Omega)$ such that $\displaystyle ||f||_{L^N(\Omega)}< \frac{1}{\mathcal{S}_1 h(\infty)}$ and let $h$ satisfy \eqref{h1} and \eqref{h2}. Then there exists a solution $u$ to problem \eqref{pb1lap} in the sense of Definition \ref{weakdefnonnegative}.  

		\end{theorem}
	
\begin{proof}
The proof strictly follows the lines of the one of Theorem \ref{existence} so we only sketch it by highlighting the main differences. One  reasons by approximation  with the distributional solutions $u_p$ to \eqref{pbplap} and use the estimates given in Lemma \ref{lemmapfisso}.  The existence of both an a.e.  limit function   $u\in BV_{\rm{loc}}(\Omega)\cap L^{\infty}(\Omega)$  and a vector field  $z\in \mathcal{D}\mathcal{M}^\infty_{\rm{loc}}(\Omega)$ with $||z||_{\infty}\le 1$ ($\ast$-weak limit in $L^{\infty}(\Omega,\mathbb{R}^N)$ of $|\nabla u_p|^{p-2}\nabla u_p$) then follows as before. Moreover $u^{\sigma}\in BV(\Omega)$  and  \eqref{nondef_campo} is in force. 
 What are left are the proofs that  $\chi_{\{u>0\}}$ belongs to   $BV_{\rm{loc}}(\Omega)$, that  $ h(u)f\in L^{1}_{\rm{loc}}(\Omega)$,  and that   \eqref{nondef_eqdistr} holds. 

 It  follows by the Fatou lemma applied to  \eqref{pweakdef} that    
	\begin{equation}\label{debole}
	 \int_{\Omega}h(u)f\varphi \leq  \int_{\Omega}z\cdot \nabla \varphi  =-\int_{\Omega}\varphi\operatorname{div}z\ \ \ \forall\varphi \in C^1_c(\Omega), \ \ \varphi \ge 0;
	\end{equation}

in particular $ h(u)f\in L^{1}_{\rm{loc}}(\Omega)$. Now we test \eqref{pbplap} with  $S_\delta(u_p)\varphi$ (see Remark \ref{remtest}\bk), where $S_\delta(s):= 1-V_\delta(s)$,  $V_\delta$ is defined in \eqref{V_delta},  and $0\le\varphi\in C^1_c(\Omega)$,   obtaining
	\begin{equation}\label{minore1}
	\int_{\Omega}|\nabla u_p|^{p} S'_\delta(u_p) \varphi + \int_{\Omega}|\nabla u_p|^{p-2} \nabla u_p\cdot \nabla \varphi S_\delta(u_p)  = \int_{\Omega}h(u_p)fS_\delta(u_p)\varphi.
	\end{equation} 
	Thus from \eqref{minore1}, using Young's inequality   (recall $p>1$),  we have
	\begin{equation}
	\begin{aligned}\label{minore2}
	&\int_{\Omega} |\nabla S_\delta(u_p)|\varphi + \int_{\Omega}|\nabla u_p|^{p-2} \nabla u_p\cdot \nabla \varphi S_\delta(u_p)
	\\ 
	&\le \frac{1}{p}\int_{\Omega} |\nabla u_p|^p S'_\delta(u_p)\varphi + \frac{p-1}{p}\int_{\Omega} S'_\delta(u_p)\varphi + \int_{\Omega}|\nabla u_p|^{p-2} \nabla u_p\cdot \nabla \varphi S_\delta(u_p)
	\\ 
	&\le \frac{p-1}{p}\int_{\Omega}  S'_\delta(u_p)\varphi +  \int_{\Omega}h(u_p)fS_\delta(u_p)\varphi.
	\end{aligned} 	
	\end{equation}
 We want to  pass to the limit as $p\to1^{+}$ first, and then we will let  $\delta\to 0^{+}$. 
 First of all, using  that $|\nabla S_\delta(u_p)| = |S'_\delta(u_p)\nabla u_p| \le\frac1\delta |\nabla u_p|$ then it follows from \eqref{exstima} the uniform local boundedness of  $S_\delta(u_p)$ in $BV_{\rm{loc}}(\Omega)$  with respect to $p$ and we can pass to the limit in \eqref{minore2} by weak lower semicontinuity in the first term on the left hand side.  Also the second term easily passes to the limit.  On the right hand side,  the first term vanishes (as $S'_{\delta}$ is bounded) while for the second  term we have 
	$$
	h(u_p)fS_\delta(u_p)\varphi\leq h(u_{p})f\varphi\chi_{\{u_{p}>\delta\}}\leq\sup_{s\in [\delta,\infty)}h(s) \ f \varphi 
	$$

	so that, by  dominated convergence theorem, we can pass to the limit in this term as well   finally get 
	\begin{align*}
	\int_{\Omega} |D S_\delta(u)|\varphi + \int_{\Omega}z\cdot \nabla \varphi S_\delta(u)\le \int_{\Omega}h(u)fS_\delta(u)\varphi.
	\end{align*} 
	Thanks to the fact that $z\in \mathcal{D}\mathcal{M}^\infty_{\rm{loc}}(\Omega)$ and to \eqref{debole}, we have that all but the first term in the previous are uniformly  bounded with respect to $\delta$. Hence $S_\delta(u)$ is bounded in $BV_{\rm{loc}}(\Omega)$ and we are allowed to pass to the limit in $\delta$ (using once again weak lower semicontinuity in  the first term) and Lebesgue's theorem for the remaining terms, getting 	
$$
	\int_{\Omega} |D \chi_{\{u>0\}}|\varphi + \int_{\Omega}z\cdot \nabla \varphi \chi_{\{u>0\}} \le \int_{\Omega}h(u)f\chi_{\{u>0\}}\varphi\,.
$$				
	Observe, in particular, that 
	\begin{equation*}
	\chi_{\{u>0\}} \in BV_{\rm{loc}}(\Omega).
	\end{equation*} 
	Therefore,  recalling Proposition \ref{nuova},  
	$$ -(\operatorname{div}z )\chi^*_{\{u>0\}} = - \operatorname{div}(z \chi_{\{u>0\}}) + (z,D \chi_{\{u>0\}}),$$
	and so 
	\begin{align}\label{minore4bis}
	-\int_{\Omega}\varphi\chi^*_{\{u>0\}}\operatorname{div}z \le \int_{\Omega}h(u)f\chi_{\{u>0\}}\varphi\le \int_{\Omega}h(u)f\varphi.
	\end{align} 
	We prove the reverse inequality.
	In \eqref{debole} we take $\varphi = (\chi_{\{u>0\}}* \rho_\epsilon)\phi$ where $0\le \phi \in C^1_c(\Omega)$ and $\rho_\epsilon$ is a mollifier. Passing to the limit in $\epsilon$ (Lebesgue's theorem on the left hand side and Fatou's lemma on the right hand side) we obtain 	
	\begin{equation}\label{maggiore}
	-\int_{\Omega}\phi\chi^*_{\{u>0\}}\operatorname{div}z \ge \int_{\Omega}h(u)f\chi_{\{u>0\}}\phi \overset{\{u=0\}\subset \{f=0\}}{=}  		\int_{\Omega}h(u)f\phi \ \ \ \forall\phi \in C^1_c(\Omega), \ \ \phi \ge 0,
	\end{equation}
	then \eqref{minore4bis} and \eqref{maggiore} imply that \eqref{nondef_eqdistr} holds.
\end{proof}

\section{The case $f$ in $ L^{N,\infty}(\Omega)$}	\label{lore}

The main results proven in  the previous section can be extended to the case of a slightly more general nonnegative datum in the Lorentz space $f\in L^{N,\infty}(\Omega)$, also called Marcinkiewicz space, this extension being optimal in the sense specified below (see Remark \ref{optimo}). We refer, for instance, to the monograph \cite{PKF} for  a smooth introduction to the subject of Lorentz spaces and their main properties. We only recall that an  H\"older's inequality is available in this case and that the conjugate space associated to $L^{p,q}(\Omega)$ for $p>1$ and $q\in [1,\infty]$ is $L^{p',q'}(\Omega)$. Also,  a  Sobolev embedding  inequality for $W^{1,p}_{0}(\Omega)$ holds, that is
\begin{equation}\label{soblor}
\|v\|_{L^{p^{\ast},p}(\Omega)}\leq \tilde{\mathcal S}_{p} \|\nabla v\|_{W^{1,p}_{0}(\Omega)}, \ \ \forall v\in W^{1,p}_{0}(\Omega)\,.
\end{equation}

The involved constants are explicit and one has
$$ \tilde{\mathcal S}_{p}=\frac{p\Gamma(1+\frac{N}{2})^{\frac{1}{N}}}{\sqrt{\pi}(N-p)} \stackrel{p\to 1^{+}}{\longrightarrow}  \tilde{\mathcal S}_{1}= [(N-1)\omega_{N}^{\frac{1}{N}}]^{-1}, \ \ $$
where $\Gamma$ is the usual Gamma function (see \cite{A, CRT}). 

\medskip 

We  consider problem 
\begin{equation} \label{plor}
	 \begin{cases}
	 \displaystyle -\Delta_{1} u = h(u)f &  \text{in}\, \Omega, \\
	 u=0 & \text{on}\ \partial \Omega,			
	 \end{cases}
\end{equation}
where $f\in L^{N, \infty}(\Omega)$ is nonnegative  and $h$, as before,  is a continuous function satisfying \eqref{h1} and \eqref{h2}. 

Definition \ref{weakdefnonnegative} can be straightforwardly re-adapted to this case with many simplifications if $f>0$ (as in  Definition \ref{weakdef}).
We summarize the results one can obtain in the following

\begin{theorem}\label{tuttolor}
	Let $0\le f\in L^{N,\infty}(\Omega)$ such that $\displaystyle ||f||_{L^{N,\infty}(\Omega)}<\frac{1}{\tilde{\mathcal{S}}_1 h(\infty)}$, where  $h$ satisfies \eqref{h1} and \eqref{h2}. Then there exists a (unique, if $h$ is decreasing and $f>0$ a.e. in $\Omega$)  solution $u$ to problem \eqref{plor} in the sense of Definition \ref{weakdefnonnegative}. 

Moreover,  if $h(0)= \infty$  then $u>0$ a.e. in $\Omega$, and,  if $h\in L^\infty([0,\infty))$ and  \begin{equation}\label{opt}||f||_{L^{N,\infty}(\Omega)}< \frac{1}{\tilde{\mathcal{S}}_1 ||h||_{L^\infty([0,\infty))}}\,,\end{equation} then $u\equiv 0$ a.e. in $\Omega$.

\end{theorem}
		
	\begin{remark}\label{optimo}
	Observe that condition   \eqref{opt} is optimal in the sense that,   if $h\equiv 1$, one can construct a datum $\overline f$ with  $ ||\overline f||_{L^{N,\infty}(\Omega)}= {\tilde{\mathcal{S}_1}^{-1} }$ such that problem \eqref{plor} relative to $\overline f$ admits a non-trivial solution (ie. $u\neq 0$)  (see \cite[Theorem 3.4, Remark 3.2]{CT}).  
	\end{remark}	

As far as the proofs of our  existence, uniqueness and regularity results in the case $f\in L^{N}(\Omega)$ are concerned, the proof of Theorem \ref{tuttolor} is a standard re-adaptation once the analogous of  Lemma \ref{lemmapfisso} is established.  That is, if we consider 
	 \begin{equation}
	 \begin{cases}
	 \displaystyle - \Delta_p u_n= h_n(u_n)f_n &  \text{in}\, \Omega, \\
	 u_n=0 & \text{on}\ \partial \Omega,
	 \label{pbintrolor}
	 \end{cases}
	 \end{equation}
	 where $f\in L^{N,\infty}(\Omega)$ is nonnegative and $f_n=T_n(f)$, $h_n(s)=T_n(h(s))$. Hence Theorem \ref{tuttolor} is a consequence of the following
\begin{lemma}\label{priorilemmalor}
	Let $u_n$ be a solution to \eqref{pbintrolor}. If $\displaystyle ||f||_{L^{N,\infty}(\Omega)}< \frac{1}{\tilde{\mathcal{S}_1} h(\infty)}$ then there exists $\overline{k}>0$ such that $u_n$ satisfies:
	\begin{equation}\label{prioriuboundedlor}
	||G_{k}(u_n)||_{L^{\frac{N}{N-1},1}(\Omega)} \le \frac{p-1}{p} |\Omega|C(\tilde{\mathcal{S}_1}, h(\infty), ||f||_{L^{N,\infty}(\Omega)},\overline{k}),\ \ \ \text{for all } k\geq \overline k,
	\end{equation}
	\begin{equation}\label{priorieqlor}
	||u_n||_{W^{1,p}(\omega)}\le C(\tilde{\mathcal{S}_1}, h(\infty), ||f||_{L^{N,\infty}(\Omega)},\omega, |\Omega|), \ \ \ \text{for all } \omega \subset\subset \Omega,
	\end{equation}	
	\begin{equation}\label{prioribordolor}
	||u_n^{\frac{\sigma-1+p}{p}}||_{W^{1,p}_0(\Omega)} \le C(\tilde{\mathcal{S}_1}, \sup_{s\in [k_0,\infty)} h(s), ||f||_{L^{N,\infty}(\Omega)}, c_1, |\Omega|)\,, 	\end{equation}			\end{lemma}\bk
\begin{proof}
The proofs of both \eqref{prioriuboundedlor} and \eqref{priorieqlor} strictly follow the ones who led to  \eqref{exstima} and \eqref{exubounded} where, systematically, the H\"older inequality is replaced by the generalized H\"older inequality in Lorentz spaces and \eqref{soblor} substitutes the usual Sobolev's embedding inequality in Lebesgue's spaces. 
The only estimate that  needs some further efforts is \eqref{prioribordolor} and we focus on it. 
As in \eqref{stimapotenza} one fixes $k_1>k_0$ and then  multiplies  \eqref{pbintrolor}  by  $u_n^\sigma$, obtaining 
	\begin{equation}	\label{stimapotenzalor}
	\begin{aligned}
	\displaystyle  \left(\frac{p}{\sigma-1+p}\right)^p\sigma\int_{\Omega}|\nabla u_n^{\frac{\sigma-1+p}{p}}|^p &= \int_{\Omega}h_n(u_n)fu_n^\sigma \le c_{1}k_0^{\sigma-\gamma}\int_{\{u_n\le k_0\}}f
	\\ 
	& + \max_{s\in [k_0,k_1]}h (s) \ k_1^\sigma \int_{\{k_0<u_n< k_1\}} f\ +\  (h(\infty)+\epsilon_{k_1})\int_{\{u_n\ge k_1\}}fu_n^\sigma
    \\  
    &\le 
 C(k_0,\max_{s\in[k_0,k_1]} h(s), ||f||_{L^{1}(\Omega)}, c_1) \bk\\	
	&+  (h(\infty)+\epsilon_{k_1})||f||_{L^{N,\infty}(\Omega)} ||u_n^\sigma||_{L^{\frac{N}{N-1},1} (\Omega)}. 	 
	\end{aligned}	
	\end{equation}
	
	\medskip 
	Now, let $u^{\ast}(t)$ be the non-increasing rearrangement of $u$ for $t\in (0,|\Omega|)$,  and observe that $(u^{q})^{\ast}= (u^*)^q$ for $q>0$.  Using  the H\"older inequality with exponent $\frac{\sigma-1+p}{\sigma}$  we have 
	\begin{align}
		\nonumber ||u_n^\sigma||_{L^{\frac{N}{N-1},1}(\Omega)} &= \int_{0}^{|\Omega|} t^{-\frac{1}{N}} u_n^*(t)^\sigma dt \le \left(\int_{0}^{|\Omega|} t^{-\frac{\sigma-1+p}{N\sigma}} u_n^*(t)^{\sigma-1+p} dt\right)^\frac{\sigma}{\sigma-1+p} |\Omega|^{\frac{p-1}{\sigma-1+p}}
		\\ \nonumber
		&= \left(\int_{0}^{|\Omega|} \left(t^{-\frac{\sigma(1-N)-1+p}{Np\sigma}} u_n^*(t)^{\frac{\sigma-1+p}{p}}\right)^{p} \frac{dt}{t}\right)^\frac{\sigma}{\sigma-1+p} |\Omega|^{\frac{p-1}{\sigma-1+p}}
		\\ \nonumber
		&= \left(\int_{0}^{|\Omega|} \left(t^{\frac{1}{p^*}} u_n^*(t)^{\frac{\sigma-1+p}{p}}\right)^{p} t^{\frac{(\sigma-1)(p-1)}{N\sigma }}\frac{dt}{t} \right)^\frac{\sigma}{\sigma-1+p} |\Omega|^{\frac{p-1}{\sigma-1+p}}
		\\ \nonumber
		&\le \left(\int_{0}^{|\Omega|} \left(t^{\frac{1}{p^*}} u_n^*(t)^{\frac{\sigma-1+p}{p}}\right)^{p} \frac{dt}{t} \right)^\frac{\sigma}{\sigma-1+p} |\Omega|^{\frac{(p-1)(\sigma-1+N)}{N(\sigma-1 +p)}}
		\\ \nonumber
		&= ||u_n^\frac{\sigma-1+p}{p}||^\frac{p\sigma}{\sigma-1+p}_{L^{p^*,p}(\Omega)} |\Omega|^{\frac{(p-1)(\sigma-1+N)}{N(\sigma-1 +p)}}
		\\ \nonumber
		&\le \frac{\sigma}{\sigma-1+p}||u_n^\frac{\sigma-1+p}{p}||^p_{L^{p^*,p}(\Omega)} 
		+ \frac{p-1}{\sigma-1+p} |\Omega|^{\frac{\sigma-1+N}{N}}\,,
	\end{align}
	where in the last step  we  used Young's inequality. 
	Concerning  the left hand side of \eqref{stimapotenzalor} we have 
	\begin{align*}
	\displaystyle  \left(\frac{p}{\sigma-1+p}\right)^p\sigma\int_{\Omega}|\nabla u_n^{\frac{\sigma-1+p}{p}}|^p \ge \displaystyle  \left(\frac{p}{\sigma-1+p}\right)^p\frac{\sigma}{ \tilde{\mathcal{S}}_p^{p}}||u_n^\frac{\sigma-1+p}{p}||^p_{L^{p^*,p}(\Omega)} \,,
	\end{align*}	 	
	and gathering  the previous two inequalities with  \eqref{stimapotenzalor} we deduce 
	\begin{align}\label{stimapotenza2lor}
	\displaystyle  \left(\left(\frac{p}{\sigma-1+p}\right)^p\frac{\sigma}{ \tilde{\mathcal{S}}_p^{p}} - (h(\infty)+\epsilon_{k_1})||f||_{L^{N,\infty}(\Omega)} \frac{\sigma}{\sigma-1+p}  \right)||u_n^\frac{\sigma-1+p}{p}||^p_{L^{p^*,p}(\Omega)}\le C, 
	\end{align}
	where, since $\displaystyle ||f||_{L^{N,\infty}(\Omega)}< \frac{1}{\tilde{\mathcal{S}}_1 h(\infty)}$, one can pick $p$ near to $1$ and $k_1$  large enough such that
	$$
	\displaystyle  \left(\left(\frac{p}{\sigma-1+p}\right)^p\frac{\sigma}{ \tilde{\mathcal{S}}_p^{p}} - (h(\infty)+\epsilon_{k_1})||f||_{L^{N,\infty}(\Omega)} \frac{\sigma}{\sigma-1+p} \right)>c>0
	$$ 
	for a constant    $c$ that does not depend on both $p$ and $k_1$. 
	Using estimate \eqref{stimapotenza2lor} in \eqref{stimapotenzalor} one finally deduces \eqref{prioribordolor}.
\end{proof}

		\section{Further extensions, remarks, and examples}\label{secfinal}
		\setcounter{equation}{0}	
		
		\subsection{Some global $BV$ solutions}\label{finite}		
	
	\label{energiafinita}
	
	As we have seen the fact $u^\sigma\in BV(\Omega)$ is crucial in order to prove the existence and uniqueness of a solution to \eqref{pb1lap}. Due to the possible degeneracy of the datum $f$ and to the weak requests one assumes on the nonlinearity $h$ this step seems to be needed, in general, in order to conclude.   
	
	\medskip

	Although, a natural question is whether the solution to \eqref{pb1lap} enjoys itself the further property to have global finite energy  in the natural space  $BV(\Omega)$. To fix the ideas, if $\Omega$ is a smooth domain,  consider the problem 
	\begin{equation}
	\label{op}				\begin{cases}
	\displaystyle - \Delta_p u= \frac{f}{u^\gamma} &  \text{in}\, \Omega, \\
	u=0 & \text{on}\ \partial \Omega,
	\end{cases}
	\end{equation}	
	where $\gamma>0$ and $f$ is an H\"older continuous function that is bounded away from zero on $\Omega$. If $p=2$ and $\gamma\geq 1$  it can be proven that solutions belong to the natural space $H^1_0(\Omega)$ if and only if $\gamma<3$ (\cite{LM}, see also \cite{SZ, OP} for further refinements). Extensions to the case $p>1$ are also available (\cite{S}) and the threshold becomes $\gamma<\frac{2p-1}{p-1}$, suggesting that, as $p\to 1^+$, one should recover the global $BV$ regularity of the solutions for any $\gamma>0$ at least for both a non-degenerate datum and a smooth domain. Recall that, as for the case $p>1$ with sufficiently integrable data, if $\gamma\leq 1$ (see \cite{DGS}) then solutions to problem \eqref{op} always have finite energy  if  $p=1$ (compare with Theorem \ref{regularity} above).  
	
	\medskip 
	
	In order to better understand this phenomenon one can look  at the proof of the result in \cite{LM}.  One immediately realizes that, at regular boundary points, the slope of the solutions to \eqref{op}  become larger and larger as $\gamma$ grows eventually leading the solution to  loose its $C^1$ regularity (beyond  $\gamma=1$) and its $H^1$ regularity (at $\gamma=3$).  Hence, 
	due to the fact that the boundary datum needs not to be attained in the classical sense, and to the particular nature of $BV$ (that allows jumps),  it seems reasonable that, for any fixed $\gamma>0$ solutions to \eqref{op} may become globally $BV$ as $p$ reaches $1$ (at least if the datum $f$ does not degenerate at zero).

	\medskip
	In this section we want to present  some further evidences of this fact. Though a more general right hand side can be considered, in order to simplify the exposition we consider the simplest model 
	\begin{equation}\label{lm}
	\begin{cases}
	\displaystyle - \Delta_1 u= {u^{-\gamma}} &  \text{in}\, \Omega, \\
	u=0 & \text{on}\ \partial \Omega
	\,.
	\end{cases}
	\end{equation}
	
	We will construct an example showing  that, for a rich enough class of domains, solutions  to \eqref{lm} belongs  to $BV(\Omega)$, for any $\gamma>0$. We first need the following

	\begin{defin}\label{calib}
		We say that a bounded convex set $E$ of  class $C^{1,1}$  is {\it calibrable} if there exists a vector field $\xi\in L^\infty(\R^N, \R^N)$ such that $\|\xi\|_\infty\leq 1$, $(\xi,D\chi_E)=|D\chi_E|$ as measures,  and $$-{\rm div}\xi=\lambda_E\chi_E \quad {\rm in \ } \mathcal D'(\R^N) $$ for some constant $\lambda_E$. In this case  $\lambda_E=\frac{Per(E)}{|E|}$ and  $[\xi,\nu^E]=-1$, $\mathcal H^{N-1}$-a.e in $\partial E$ (see \cite[Section 2.3]{acc} and \cite{MP}).
	\end{defin}
	
	As a consequence of   \cite[Theorem 9]{acc}  a bounded and convex set $E$ is calibrable if and only if the following condition holds: $$(N-1)\|{\bf H}_E\|_{L^\infty(\partial E)}\leq \lambda_E=\frac{ Per(E)}{|E|},$$ where ${\bf H}_E$ denotes the ($\mathcal H^{N-1}$-a.e. defined) mean curvature of $\partial E$. In particular, if $E=B_R(0)$, for some $R>0$, then $E$ is calibrable.
	
	\begin{example}\label{ex1}
		 If $\Omega$ is a calibrable set, let us prove that $u=\left(\frac{|\Omega|}{Per(\Omega)}\right)^{\frac{1}{\gamma}}$ is the unique  solution to \eqref{lm} in the sense of Definition \ref{weakdef}. It suffices to take the restriction to $\Omega$ of the vector field in the definition of calibrability; i.e.: $z:=\xi_{\res_\Omega}$. In fact, due to the properties of  $\xi$ one has
		\begin{equation}\label{usi}
		-{\rm div}z=\frac{Per(\Omega)}{|\Omega|}=u^{-\gamma}\ \ \ \text{and}\ \ \  [\xi,\nu^{\Omega}]=-1\,.
		\end{equation}
		
		Moreover, using both   \eqref{GreenIII} \bk and \eqref{usi}, one finally gets 
		$$\begin{array}{l}\dys  (z,Du)(\Omega)=\int_\Omega \left(\frac{|\Omega|}{Per(\Omega)}\right)^{\frac{1}{\gamma}}\frac{Per(\Omega)}{|\Omega|}\,dx\\\\\dys +\int_{\partial\Omega}[\xi,\nu^\Omega]\left(\frac{|\Omega|}{Per(\Omega)}\right)^{\frac{1}{\gamma}}\,d\mathcal H^{N-1}=0= |Du|(\Omega).\end{array}$$
	\end{example}

	\begin{remark}
		
		Observe that the solutions of \eqref{lm} given by Theorem \ref{existence} belong to $BV(\Omega)$ once they  are bounded away from zero on $\Omega$; in fact, let $u\geq a>0$, then,  due to property \eqref{eqregul} one has $
		u=S(u^\gamma)\in BV(\Omega),
		$	
		where $S$ is the Lipschitz continuous function defined by
		$$
		S(s)=\max{(a, s^\frac{1}{\gamma} )}\,.
		$$
		Let us show a situation in which  $u\ge a>0$ holds.  
		Let $\Omega$ be a convex open set. In \cite{MP} it is shown that $-H_{\Omega} (x)$ is a (so called) large solution to $\Delta_1 v= v $, i. e. 
		\begin{equation}\label{large}			\begin{cases}
		\displaystyle \Delta_1 v= v &  \text{in}\, \Omega, \\
		v=\infty & \text{on}\ \partial \Omega,			\end{cases}
		\end{equation}
		where $H_{\Omega} (x)$ is the variational mean curvature of $\Omega$ (see \cite{BGT} for details). Without entering into technicalities, only recall that $\|H_\Omega\|_{L^{\infty}(\rn)}< \infty$ if and only if $\Omega$ is of class $C^{1,1}$; in particular,  these  solutions only assume the (large) datum $\infty$  at non-regular points of $\Omega$ (e.g. at corners).  
		
		As through the change of variable $u=v^{-\frac{1}{\gamma}}$ problem \eqref{large} formally transforms into \eqref{lm} (using also the homogeneity of the operator) then one can expect that solutions to  problem \eqref{lm} always belong to $BV(\Omega)$ if $\Omega$ is a convex bounded  $C^{1,1}$ domain and that, in general,  the Dirichlet homogeneous boundary datum is only assumed pointwise at  non-smooth points of $\Omega$. 	
	\end{remark}

			\subsection{More general growths}	\label{more}

			Here we  show how the assumption on the control of $h$ near zero can be removed allowing more general growths not satisfying \eqref{h1}. 
Consider the general  problem 			
			\begin{equation} \label{pb1lapfs}
			\begin{cases}
			\displaystyle - \Delta_{1} u = F(x,u) &  \text{in}\, \Omega, \\
			u=0 & \text{on}\ \partial \Omega,			
			\end{cases}
			\end{equation}
			where $F(x,s)$ is a nonnegative Carath\'eodory function satisfying 
\begin{equation}\label{controlloF}
	F(x,s)\leq h(s)f(x),\ \ \forall\ (x,s)\in \Omega\times [0,\infty)\,,
\end{equation}
with $0 \le f\in L^{N,\infty}(\Omega)$,  and $h$ is a continuous function in $[0,\infty)$ satisfying \eqref{h2}. \\
First of all, without loss of generality,  we can assume that $h$ such that
\begin{equation}\label{hdecreasing}
	h \text{ is decreasing}, h \in C^{1}((0,\infty)), h^{-1}\in C^1([0,\infty)), h^{-1}(0)=0\,,
\end{equation}
where $h^{-1}$ stands for the reciprocal of $h$.  Indeed, for any given $h$ satisfying our assumptions one can construct (see for instance \cite[Remark 2.1, vii)]{GMM}) a function $\overline h$ such that \eqref{hdecreasing} holds and 
$$h(s)\leq \overline h (s), \ \ \forall s\geq0.$$
As for the previous sections we look for  a solution of \eqref{pb1lapfs} through an approximation argument,  letting $p\to 1^+$ in the solutions to
\begin{equation}\label{pbplapfs}
	\begin{cases}
	\displaystyle - \Delta_p u_p= F(x,u_p) &  \text{in}\, \Omega, \\
	u_p=0 & \text{on}\ \partial \Omega.
	\end{cases}
\end{equation}	
The notion of solution to \eqref{pbplapfs} for $p> 1$ is  the following one.
 \begin{defin}\label{distributionalgeneral}
	A nonnegative function $u_p\in W^{1,p}_{\rm{loc}}(\Omega)$ is a \emph{distributional solution} to problem \eqref{pbplapfs} if 
	\begin{equation}
	F(x,u_p)\in L^1_{\rm loc}(\Omega), \label{pweakdefb1}
	\end{equation}
	\begin{equation}
	G_{k}(u_{p})\in W^{1,p}_{0}(\Omega), \ \ \ \text{for all } k>0,  \label{pweakdef2}
	\end{equation}	
	and
	\begin{equation}
	\int_{\Omega}|\nabla u_p|^{p-2} \nabla u_p\cdot \nabla \varphi = \int_{\Omega}F(x,u)\varphi, \label{pweakdefb}
	\end{equation}
	for every $\varphi \in C^1_c(\Omega)$\bk.	
\end{defin}	
We have the following result whose proof, using \eqref{controlloF},  easily follows line by line the proof of Theorem \ref{existence_p}.  Only observe that, in this case, the approximating problems read as
	 \begin{equation}
	 \begin{cases}
	 \displaystyle - \Delta_p u_n= F_{n}(x,u_{n}) &  \text{in}\, \ \Omega, \\
	 u_n=0 & \text{on}\ \partial \Omega, \label{pbintrolorn}
	 \end{cases}
	 \end{equation}
	 where $F_{n}(x,s)= F(x,T_{n}(s))$.  Moreover we set
\begin{equation}\label{beta}
	\beta_p(s)=\int_{0}^{s} ((h^{-1}(t))')^{\frac{1}{p}}dt\,.
\end{equation}	  
\begin{theorem}\label{existence_pgeneral}
	Let $F$ satisfy \eqref{controlloF} where $0\le f\in L^{(p^*)'}(\Omega)$ and let $h$ satisfy \eqref{hdecreasing} and \eqref{h2}. Then there exists a solution $u_p$ to problem \eqref{pbplapfs} in the sense of Definition \ref{distributionalgeneral} such that $\beta_p(u_p) \in W^{1,p}_0(\Omega)$.
\end{theorem}

The following counterpart of \eqref{dispotenza} can be proven:  
\begin{equation}\label{dispotenzagen}
	\int_{\Omega} |\nabla \underline{\Gamma}_{p}(u_{p})|^p \le  \int_{\Omega}  F(x,u_p) \underline{\psi}_p(u_{p}),	
\end{equation}
where, for $\delta> 0$, 
\begin{equation*}\label{psigeneral}
	\underline{\psi}_p(s)= 
	\begin{cases}
	h^{-1}(s)\frac{\beta_p(\delta)}{h^{-1}(\delta)} \ \ \ &\text{if }s<\delta, \\
	\beta_p(s)   &\text{if }s\ge\delta\,,
	\end{cases}
\end{equation*}
and $\underline\Gamma_{p}$ is the primitive of $\underline{\psi}'^{\frac{1}{p}}_p(s)$ (such that $\underline\Gamma_{p}=0$).

Here is how the definition of solution to problem \eqref{pb1lapfs} can be suitably modified:
\begin{defin}\label{weakdefnonnegativefs}
	 A function $u\in BV_{\rm{loc}}(\Omega)\cap L^\infty(\Omega)$ having $\chi_{\{u>0\}} \in BV_{\rm{loc}}(\Omega)$ and $\beta_1(u) \in BV(\Omega)$ ($\beta_1$ is defined in \eqref{beta}) is a solution to problem \eqref{pb1lapfs} if there exists $z\in \mathcal{D}\mathcal{M}^\infty_{\rm{loc}}(\Omega)$ with $||z||_{\infty}\le 1$ such that
	\begin{align}
	&F(x,u)\in L^1_{\rm{loc}}(\Omega), 
	\\
	&-(\operatorname{div}z)\chi^*_{\{u>0\}} = F(x,u) \ \ \ \text{in  } \mathcal{D'}(\Omega), \label{nondef_eqdistrfs}
	\\
	& (z,D u)=|D u| \label{nondef_campofs} \ \ \ \ \text{as measures in } \Omega,
	\\
	&  \beta_1(u(x)) + [\beta_1(u) z,\nu] (x)= 0 \label{nondef_bordofs}\ \ \ \text{for  $\mathcal{H}^{N-1}$-a.e. } x \in \partial\Omega.
	\end{align}
\end{defin}		
	
One finally has the following 
\begin{theorem}\label{existencenonnegativefs}
	Let $0\le f\in L^{N,\infty}(\Omega)$ such that $\displaystyle ||f||_{L^{N,\infty}(\Omega)}< \frac{1}{\tilde{\mathcal{S}_1} h(\infty)}$ and let $h$ satisfy \eqref{h2}. Then there exists a solution $u$ to problem \eqref{pb1lapfs} in the sense of Definition \ref{weakdefnonnegativefs}. Moreover
	if $F(x,0)= \infty$ then $u>0$ a.e. in $\Omega$. Otherwise if $F(x,0)< \infty$ and if $||f||_{L^{N,\infty}(\Omega)}< (\tilde{\mathcal{S}_1} ||h||_{L^\infty([0,\infty))})^{-1}$ then $u\equiv 0$ a.e. in $\Omega$. 
\end{theorem}
\begin{proof}

The proof is a suitable modification of the one of Theorems \ref{existence},  \ref{existencenonnegative}, and \ref{tuttolor};  we only highlight the main differences. 
One starts with the solutions $u_{n}$ of \eqref{pbintrolorn}. As $f\in L^{N,\infty}(\Omega)$ with  $\dys ||f||_{L^{N,\infty}(\Omega)} < \frac{1}{\tilde{\mathcal{S}}_{1} h(\infty)}$, then, uniform estimates hold. In fact,  using \eqref{controlloF},  both  \eqref{prioriuboundedlor} and \eqref{priorieqlor}  continue to hold,  and, recalling \eqref{hdecreasing},  one can show  that $p_{0}$ exists such that for any $p\in (1,p_0)$ 
\begin{equation}\label{stimabeta}
	||\beta_p(u_n)||_{W^{1,p}_0(\Omega)} \le C(\tilde{\mathcal{S}}_1, \sup_{s\in [1,\infty)}h(s), ||f||_{L^{N,\infty}(\Omega)}).
\end{equation}
This is done by considering the solutions to \eqref{pbintrolorn} and   taking $h^{-1}(u_{n})$ as test.   
 By weak lower semicontinuity the same holds for  $u_p$. One then deduces, reasoning as in the proof of Lemma  \ref{lemma_campoz},   the existence of a vector field  $z\in \mathcal{D}\mathcal{M}^\infty_{\rm{loc}}(\Omega)$ with $||z||_{\infty}\le 1$ which is the $\ast$-weak limit of $|\nabla u_p|^{p-2}\nabla u_p$. The proof of \eqref{nondef_eqdistrfs} can be derived as for \eqref{nondef_eqdistr}.
Now recalling \eqref{stimabeta}, we have
\begin{equation*}
	\int_{\Omega} |\nabla \beta_p(u_p)| \le \frac{1}{p}\int_{\Omega} |\nabla \beta_p(u_p)|^p +\frac{1}{p'} |\Omega|\le C,  
\end{equation*} 
which implies that $\beta_p(u_p)$ is bounded in $BV(\Omega)$ and then it strongly converges  in $L^q(\Omega)$ to its a.e. limit $\beta_1(u)\in BV(\Omega)$. Moreover $\nabla \beta_p(u_p)$ converges $\ast$-weakly in the sense of measures to $D\beta_1(u)$. 
Hence reasoning exactly as in the proof of \eqref{equ} it yields 
\begin{equation}\label{equbeta}
	 -\beta_1(u)^*\operatorname{div}z =  F(x,u) \beta_1(u) \text{  in  } \mathcal{D'}(\Omega).
\end{equation}			 
In order to prove \eqref{nondef_campofs}, one  lets $0\le \varphi \in C^1_c(\Omega)$ and $\beta_p(u_p)\varphi$ to  test \eqref{pbplapfs} obtaining
\begin{equation*}	
	 \int_{\Omega} \varphi|\nabla u_p|^{p}\beta'_p(u_p) + \int_{\Omega} \beta_p(u_p)|\nabla u_p|^{p-2}\nabla u_p \cdot \nabla \varphi = \int_{\Omega}  F(x,u_p) \beta_p(u_p) \varphi.	
\end{equation*}
Thus, by  Young's inequality, one deduces
\begin{align*}
	\int_{\Omega} \varphi|\nabla u_p\beta'_p(u_p)^{\frac{1}{p}}| +\int_{\Omega} \beta_p(u_p)|\nabla u_p|^{p-2}\nabla u_p \cdot \nabla \varphi\le \int_{\Omega}  F(x,u_p) \beta_p(u_p) \varphi + \frac{p-1}{p}\int_{\Omega}\varphi.	
\end{align*}
Reasoning as in the proof of Lemma \ref{lemma_campoz} all but the first term are shown to  pass to the limit with respect to $p$. Then by lower semicontinuity we  obtain, recalling \eqref{equbeta},  
\begin{equation*}
	\int_{\Omega} \varphi|D \beta_1(u)| + \int_{\Omega} \beta_1(u) z \cdot \nabla \varphi \le \int_{\Omega}  F(x,u) \beta_1(u) \varphi = -\int_{\Omega}(\beta_1(u))^*\varphi \operatorname{div}z, \ \ \ \forall \varphi\in C^1_c(\Omega), \ \ \varphi \ge 0\,,	
\end{equation*}
and we can apply  Proposition \ref{nuova} to deduce 
\begin{equation*}
	\int_{\Omega} \varphi|D \beta_1(u)| \le - \int_{\Omega} \beta_1(u) z \cdot \nabla \varphi -\int_{\Omega} \beta_1(u)^* \varphi\operatorname{div}z = \int_{\Omega}\varphi (z, D \beta_1(u)),  \ \ \ \forall \varphi\in C^1_c(\Omega), \ \ \varphi \ge 0\,,	
\end{equation*}			 	 	 then 
\begin{equation*}		
	\int_{\Omega} \varphi|D \beta_1(u)|  \le \int_{\Omega}\varphi (z, D \beta_1(u)),  \ \ \ \forall \varphi\in C^1_c(\Omega), \ \ \varphi \ge 0 \,,
\end{equation*}	 
that implies $|D \beta_1(u)| = (z, D \beta_1(u))$ as  $||z||_{\infty}\le 1$. Now since $\beta_1(s)$ is locally Lipschitz one obtains \eqref{nondef_campofs} by the same argument as in the last step of the proof of Lemma \ref{lemma_campoz}. 

It is left to prove \eqref{nondef_bordofs}. It follows by using \eqref{dispotenzagen} and Young's inequality that 
\begin{equation*}
	\int_{\Omega} |\nabla \underline{\Gamma}_p (u_p)|  + \int_{\partial \Omega}\beta_p(u_p) d\mathcal{H}^{N-1}  \le  \int_{\Omega}  F(x,u_p) \underline{\psi}_p(u_{p}) +\frac{p-1}{p}|\Omega|,	
\end{equation*}
then by  weak lower semicontinuity one can use  Gauss-Green formula \eqref{GreenIII} and   \eqref{equbeta} to have 
\begin{equation*}
	\int_{\Omega} |D \beta_1(u)| + \int_{\partial \Omega}\beta_1(u) d\mathcal{H}^{N-1} \le  \int_{\Omega}(z,D \beta_1(u)) - \int_{\partial \Omega} [\beta_1(u) z,\nu]d\mathcal{H}^{N-1},
\end{equation*}		 
that, as   $|D \beta_1(u)| = (z,D \beta_1(u))$,  gives \eqref{nondef_bordofs}.
\end{proof}

\end{document}